\documentclass[12pt]{amsart}
\usepackage{amssymb}
\usepackage{graphics}
\usepackage{latexsym}
\usepackage{amsmath}
\usepackage{amssymb}
\usepackage{amscd}
\usepackage{multicol}
\usepackage[cmtip,matrix,arrow]{xy}
\usepackage{hyperref}

\newtheorem{thm}{Theorem}[section]
\newtheorem{prop}[thm]{Proposition}
\newtheorem{cor}[thm]{Corollary}
\newtheorem{lem}[thm]{Lemma}

\theoremstyle{definition}

\newtheorem{ex}[thm]{Example}

\theoremstyle{remark}
\newtheorem{rem}[thm]{Remark}

\newcommand{\RR}{\mathbb R}
\newcommand{\ZZ}{\mathbb Z}

\newcommand{\CC}{\mathbb C}

\newcommand{\KK}{\mathbb K}

\newcommand{\vk}{\overline{k}}

\newcommand{\A}{{\mathcal A}}

\newcommand{\Fl}{{\mathrm{F\ell}}}

\newcommand{\be}{\A, \vk}
\newcommand{\bet}{\tilde{\beta}}

\begin{document}

\address{Department of Mathematics and Statistics\\ University of Regina\\ Canada}

\email{liviu.mare@gmail.com}

\title{Flag manifolds, spaces of frames, and connectedness properties}

\author{Augustin-Liviu Mare}


\begin{abstract} For integers $1\le k < n$ and real numbers $c_1, \ldots, c_k>0$ and $d_1, \ldots, d_n\ge 0$, we investigate the space 
of all $k\times n$ matrices $F$ such that the product $FF^*$ is equal to the diagonal matrix ${\rm Diag}(c_1, \ldots, c_k)$ and the squared norms of the columns of $F$ are equal to $d_1, \ldots, d_n$ respectively. Depending on the field where the coefficients of $F$ are taken from, that is, $\RR$ or $\CC$, we are mainly interested in determining whether the resulting space is path-connected or even simply connected relative to the subspace topology in ${\rm M}_{k,n}(\RR)$ or ${\rm M}_{k,n}(\CC)$. The criteria presented are closely related to results previously obtained by Cahill, Mixon, and Strawn (2017), Needham and Shonkwiler (2021), the last two authors together with Caine (2026), and the author of this work (2024 and 2026). Flag manifolds, namely  orbits of the canonical conjugation actions of ${\rm O}(n)$ and ${\rm U}(n)$ on the spaces of symmetric and Hermitian $n\times n$ matrices, respectively, play a central role in our development.\end{abstract}
\maketitle

\section{The main results} \label{sec:mainres}

Let $1\le k <n$ be integers and let $\KK$ be $\RR$ or $\CC$. A {\it frame} of $n$ vectors in $\KK^k$ is a $k\times n$ matrix whose columns span $\KK^k$ over $\KK$. For background material concerning this notion we refer to the monograph \cite{Wa}. 
 In this paper we study the space of frames with prescribed frame operator and prescribed norms of the columns. More specifically,  let $c_1, \ldots, c_k $ be strictly positive and  $d_1, \ldots, d_n$ non-negative real numbers. 
Denote by $\vec{d}$ the vector $(d_1, \ldots, d_n)$ and by $S$ the diagonal\footnote{In theory, the space ${\mathcal F}_{S, \vec{d}}^\KK$ can be defined for any positive-definite symmetric, respectively Hermitian, matrix $S$; however, by orthogonal, respectively unitary, diagonalization, its study reduces immediately to the case in which $S$ is diagonal.} matrix ${\rm Diag}(c_1, \ldots, c_k)$.  
By  ${\mathcal F}_{S, \vec{d}}^\KK$ we denote the space consisting of all $k\times n$ matrices $F=[f_1 | \ldots | f_n]$, where
$f_1, \ldots, f_n \in \KK^k$ are column vectors, such that
\begin{align*}
{}& FF^*=S,\\
{}& \|f_1\|^2=d_1, \ldots , \|f_n\|^2=d_n,
\end{align*}
$F^*$ being the adjoint matrix of $F$ (i.e., the transposed for $\KK=\RR$, respectively transposed followed by complex conjugate for $\KK=\CC$), and $\| \cdot \|$ the norm associated with the standard Euclidean, respectively Hermitian, inner product. 

We are interested here  in the topology of various classes of spaces of the type ${\mathcal F}_{S, \vec{d}}^\KK$ relative to the obvious subspace topology.   Historically speaking, the first motivation of our work is the proof of the so-called ``frame homotopy conjecture", obtained by Cahill, Mixon, and Strawn in \cite{CMS}, which asserts that if $S={\rm I}_k$ and $\|f_1\|=\cdots = \|f_n\|$, then the resulting ${\mathcal F}^\CC$ is always path-connected and the resulting ${\mathcal F}^\RR$ is path-connected if $2\le k \le n-2$. Refinements and generalizations were afterwards obtained by Needham and Shonkwiler in \cite{NS}, the same two authors together with Caine in \cite{CNS}, and Mare in \cite{Ma1} and \cite{Ma2}.

In \cite{Ma1} was investigated the space ${\mathcal F}^\RR_{{\rm I}_k, \vec{d}}$, whose elements are called Parseval frames.  It turns out that this space is non-empty exactly when $\vec{d}=(d_1, \ldots, d_n)$  lies  in the polytope $\Delta_{n,k}$ described by the conditions $0\le d_i \le 1$, $1\le i \le n$,  and $d_1+\cdots + d_n=k$. The first main result of this paper is as follows:

\begin{thm} \label{fdn} Assume that $\vec{d}\in \Delta_{n,k}$ satisfies the following condition:
\begin{equation} \label{eqndi}d_{i_1} + \cdots + d_{i_{n-k}}\ge 1 \ {\it for \ any \ pairwise \ distinct \  }  i_1, \ldots, i_{n-k} \in \{ 1, \ldots, n\}.\end{equation} Then  ${\mathcal F}^\RR_{{\rm I}_k, \vec{d}}$ is a path-connected subspace of 
${\rm M}_{k\times n}(\RR)$.
\end{thm}

This generalizes the theorem of Cahill, Mixon, and Strawn mentioned above, and also the connectedness criteria obtained in \cite[Sect. 6]{Ma1}. It is also related to results found in \cite{CNS}: see Ex.~\ref{ex:last} below. 

Next, we let $S$ vary and set the numbers $d_1, \ldots, d_n$ to be equal to each other (their common value is then automatically equal to ${\rm Tr}(S)/n$). Denote the resulting space by ${\mathcal F}^\KK(S, n)$. We will give examples showing that this space is    not connected in general for $\KK=\RR$, see Examples \ref{ex:nonconn} and \ref{ex:nonconn0}  below. In the next theorem we present  a special choice of $S$ for which the connectedness result does hold:

\begin{thm} \label{thmconn1} Take 
$$S={\rm Diag}(\underbrace{1 , \dots, 1}_{k_1},
\underbrace{2, \dots, 2}_{k_2},
\dots,
\underbrace{{r-1}, \dots, {r-1}}_{k_{r-1}}),$$
where $r\ge 2$ and  $k_1+\cdots + k_{r-1}=k$. 
Assume that $k_i\ge 2$ for all $1\le i \le r-1$ and that $n-k\ge 2$. Then ${\mathcal F}^\RR(S, n)$ is a path-connected
  subspace of ${\rm M}_{k\times n}(\RR)$.
\end{thm}

Again, this is a generalization of the theorem of Cahill, Mixon, and Strawn \cite{CMS}.

The space ${\mathcal F}^\CC_{S, \vec{d}}$ is path-connected (or empty) for any $S$ and $\vec{d}$, by a result of 
Needham and Shonkwiler, see \cite{NS}. However, ${\mathcal F}^\CC(S, n)$ is not always simply connected, as we will show in Ex.~\ref{ex:nonsconn}. Here, however, is a class of spaces of this type that are simply connected:

\begin{thm} \label{thmconn2} Take 
$$S={\rm Diag}(\underbrace{1 , \dots, 1}_{k_1},
\underbrace{2, \dots, 2}_{k_2},
\dots,
\underbrace{{r-1}, \dots, {r-1}}_{k_{r-1}}),$$
where $r\ge 2$ and $k_1+\cdots + k_{r-1}=k$.
Assume that $k_i\ge 2$ for all $1\le i \le r-1$ and that $n-k\ge 2$. Then ${\mathcal F}^\CC(S, n)$ is a simply connected
 subspace of ${\rm M}_{k\times n}(\CC)$.
\end{thm}

The strategy we will use in proving these theorems can be outlined as follows. We address the general case, of the frame space
${\mathcal F}_{S, \vec{d}}^\KK$.  First write 
$$\{c_i \mid 1 \le i \le k\} = \{\lambda_1, \ldots, \lambda_{r-1}\}$$
where $r\ge2$ and $\lambda_1, \ldots, \lambda_{r-1}$ are pairwise distinct.  Furthermore,
$$S={\rm Diag}(\underbrace{\lambda_1 , \dots, \lambda_1}_{k_1},
\underbrace{\lambda_2, \dots, \lambda_2}_{k_2},
\dots,
\underbrace{\lambda_{r-1}, \dots, \lambda_{r-1}}_{k_{r-1}}),$$
where $k_1, \ldots, k_{r-1}$ are positive integers whose sum is equal to $k$. 
For uniformity, we denote by ${\rm U}(n,\KK)$ the group of orthogonal matrices, if $\KK=\RR$, respectively of unitary matrices, if $\KK=\CC$. 
Consider 
$${\rm F}\ell^\KK={\rm F}\ell_{k_1,\ldots, k_{r-1}, n-k}^\KK
= {\rm U}(n,\KK) .{\rm Diag}(\underbrace{\lambda_1 , \dots, \lambda_1}_{k_1},
\dots,
\underbrace{\lambda_{r-1}, \dots, \lambda_{r-1}}_{k_{r-1}}, \underbrace{0 , \dots, 0}_{n-k}).$$
        The group action is in both cases given by matrix conjugation and the orbits lie in the space of symmetric matrices with real entries,
        respectively the space of Hermitian matrices (with complex entries). The orbit, known under the generic name of {\it flag manifold}\footnote{The ${\rm U}(n,\KK)$-orbit, denoted generically by $\Fl^\RR$ or $\Fl^\CC$, may vary from one section to another; its meaning will always be clear from the context.}, is diffeomorphic to the space of all sequences $V_1,\ldots, V_r$ of $\KK$-linear vector spaces of $\KK^n$, of dimensions $k_1, \ldots, k_{r-1}, n-k$, respectively, which are pairwise orthogonal relative to the standard Euclidean, respectively Hermitian inner product. As a manifold, it is diffeomorphic to the homogenenous space ${\rm U}(n, \KK)/\left({\rm U}(k_1, \KK)\times \cdots \times {\rm U}(k_{r-1}, \KK)
        \times {\rm U}(n-k, \KK)\right)$. 
        
        Denote $G:={\rm U}(k_1, \KK)\times \cdots \times {\rm U}(k_{r-1}, \KK)$, along with the obvious $G$-principal bundle       $$p:{\rm U}(n, \KK)/
         {\rm U}(n-k, \KK) \to {\rm U}(n, \KK)/\left({\rm U}(k_1, \KK)\times \cdots \times {\rm U}(k_{r-1}, \KK)
        \times {\rm U}(n-k, \KK)\right).$$
        The domain of this map can be canonically identified with the Stiefel manifold ${\rm V}_k(\KK^n)$  of all systems of $k$ orthonormal vectors in $\KK^n$.  Let $\mu_\KK: {\rm F}\ell^\KK \to \KK^n$ be the map which attaches to any matrix in 
        ${\rm F}\ell^\KK$ its diagonal components. It follows that the space ${\mathcal F}_{S, \vec{d}}^\KK$ is non-empty if and only if the pre-image 
 $\mu_\KK^{-1}(\vec{d})$ is non-empty too\footnote{Equivalently, $\vec{d}$ lies in the polytope (permutahedron) in $\RR^n$ whose vertices are the vectors of entries equal to  
 $\lambda_1$ ($k_1$ times),$\ldots, \lambda_{r-1}$ ($k_{r-1}$ times), and $0$ ($n-k$ times), these entries being arbitrarily ordered. The polytope $\Delta_{n,k}$ mentioned before in connection with ${\mathcal F}_{{\rm I}_k,\vec{d}}$ is just a special case of this.}. 
    The key ingredient we will be using is an embedding of ${\mathcal F}_{S,\vec{d}}^\KK$ into ${\rm V}_k(\KK^n)$, as a $G$-invariant subspace, such that its image under $p$ is equal to the pre-image $\mu_\KK^{-1}(\vec{d})$ (the details are presented in  Sec.~\ref{sec:relations} below). 
The resulting $G$-principal bundle ${\mathcal F}_{S,\vec{d}}^\KK \to  \mu^{-1}(\vec{d})$ is an important tool. 
First recall that if $X$ and $Y$ are topological spaces and $m\ge 1$ an integer,  a continuous map $f: X \to Y$ is $m$-connected 
if the induced maps $f_*: \pi_i(X) \to \pi_i(Y)$ are bijective for all $0\le i \le m-1$ and $f_*:\pi_m(X) \to \pi_m(Y)$ is surjective. The main tool in proving the three theorems above is the following:

\begin{thm}\label{maininstr} 
(a)   Assume that $\mu_\RR^{-1}(\vec{d})$ is non-empty and the inclusion map  $\mu_\RR^{-1}(\vec{d}) \to {\rm F}\ell^\RR$ is 1-connected. Then the space ${\mathcal F}_{S,\vec{d}}^\RR$ is a path-connected subspace of ${\rm M}_{k,n}(\RR)$.

(b)   Assume that $\mu_\CC^{-1}(\vec{d})$ is non-empty and the inclusion map  $\mu_\CC^{-1}(\vec{d}) \to {\rm F}\ell^\CC$ is 3-connected. Then the space ${\mathcal F}_{S,\vec{d}}^\CC$ is a simply connected  subspace of ${\rm M}_{k,n}(\CC)$.
\end{thm}

In the next section we will prove this theorem. We will subsequently use it to prove as follows: Thm.~\ref{thmconn2} in Sect.~\ref{complex} and Theorems~\ref{fdn} and \ref{thmconn1} in Sect.~\ref{realflag}.

\section{Relations with $\mu_\KK^{-1}(\vec{d})$} \label{sec:relations}

\subsection{The case $\KK=\RR$}\label{subsec:r}

In this subsection we will prove
Thm.~\ref{maininstr} (a). 
 Let $n, r, k, k_1, \ldots, k_{r-1}$, $S$, and $\vec{d}$  be like in Sect.~\ref{sec:mainres}. 

As usual, ${\rm V}_k(\RR^n)$ represents the Stiefel manifold of all orthonormal systems of $k$ vectors in $\RR^n$. It can be described as the space of all $k\times n$ matrices $A$ such that $AA^t={\rm I}_k$. By ${\rm F}\ell^\RR$  we have already denoted the ${\rm O}(n)$-conjugation orbit of 
$${\rm Diag}(\underbrace{\lambda_1\dots, \lambda_1}_{k_1},
\dots,
\underbrace{\lambda_{r-1}, \dots, \lambda_{r-1}}_{k_{r-1}}, \underbrace{0,\cdots, 0}_{n-k}).$$
Also recall that $$G:={\rm O}(k_1) \times {\rm O}(k_2) \times \cdots \times {\rm O}(k_{r-1}),$$
which acts on ${\rm V}_k(\RR^n)$ by matrix multiplication from the left.

\begin{lem} \label{gbundle}
 The map $p:{\rm V}_k(\RR^n) \to {\rm F}\ell^\RR$, $A\mapsto A^tSA$, is well defined and a $G$-principal bundle.
\end{lem}

\begin{proof} 
Consider the embedding ${\rm O}(n-k)\subset {\rm O}(n)$ which identifies each $B\in {\rm O}(n-k)$ with
$$\begin{pmatrix}
{\rm I}_k & 0  \\
0 & B 
\end{pmatrix}. $$
The homogeneous space ${\rm O}(n-k)\backslash {\rm O}(n)$ can be identified with the Stiefel manifold ${\rm V}_k(\RR^n)$ by taking any coset ${\rm O}(n-k)Q$ with $Q\in {\rm O}(n)$, writing
$$Q=\begin{pmatrix}
A  \\
M 
\end{pmatrix}
$$
with $A\in {\rm M}_{k,n}(\RR)$ and $M \in {\rm M}_{n-k, n}(\RR)$, and attaching $A$. The resulting assignment is well defined and bijective.

Now, the map $p: {\rm O}(n-k)\backslash {\rm O}(n) \to {\rm O}(n)/(G \times {\rm O}(n-k))$,
$$p({\rm O}(n-k)Q)=Q^{-1}(G\times {\rm O}(n-k)), \quad Q\in {\rm O}(n),$$
is a $G$-principal bundle\footnote{In general, if $H$ is a compact Lie group and $K\times L$ a closed subgroup, then the map
$L\backslash H \to H/(K\times L)$, $Lh \mapsto h^{-1}(K\times L)$, $h\in H$, is a $K$-principal bundle.}.
But ${\rm O}(n)/(G \times {\rm O}(n-k))$ is clearly diffeomorphic to the orbit
$${\rm O}(n).\begin{pmatrix}
S& 0  \\
0 & 0 
\end{pmatrix},$$
that is, to $\Fl^\RR$. 

In this way, we have constructed a $G$-principal bundle $p: {\rm V}_k(\RR^n) \to \Fl^\RR$, which can be described as follows.
Take $A\in {\rm V}_k(\RR^n)$, regarded as a $k\times n$ matrix. Pick $M\in {\rm M}_{n-k, n}(\RR)$ such that 
$$\begin{pmatrix}
A  \\
M 
\end{pmatrix}
\in {\rm O}(n)$$
(this can be done by completing the rows of $A$ to an orthonormal basis of $\RR^n$). Then
$$A=
 {\rm O}(n-k) \begin{pmatrix}
A  \\
M 
\end{pmatrix}
$$
is mapped by $p$ to
$$
\begin{pmatrix}
A  \\
M 
\end{pmatrix}^{-1} .(G\times {\rm O}(n-k))
=\begin{pmatrix}
A  \\
M 
\end{pmatrix}^{-1} 
\begin{pmatrix}
S& 0  \\
0 & 0 
\end{pmatrix}
\begin{pmatrix}
A  \\
M 
\end{pmatrix}
=\begin{pmatrix}
A  \\
M 
\end{pmatrix}^{t} 
\begin{pmatrix}
S& 0  \\
0 & 0 
\end{pmatrix}
\begin{pmatrix}
A  \\
M 
\end{pmatrix}
=A^tSA.$$
\end{proof}

Denote now by $\tilde{\mathcal F}$  the space of all $A\in {\rm V}_k(\RR^n)$ such that 
the diagonal of $A^tSA $ is $(d_1,\ldots, d_n)$. Consider 
$$S^{-1/2}:={\rm Diag}(\underbrace{\lambda_1^{-1/2}\dots, \lambda_1^{-1/2}}_{k_1},
\underbrace{\lambda_2^{-1/2}, \dots, \lambda_2^{-1/2}}_{k_2},
\dots,
\underbrace{\lambda_{r-1}^{-1/2}, \dots, \lambda_{r-1}^{-1/2}}_{k_{r-1}}).$$

\begin{lem} \label{lem:homeo} The map $F\mapsto S^{-1/2}F$ is a homeomorphism between ${\mathcal F}_{S, \vec{d}}^\RR$ and $\tilde{\mathcal F}$. 
\end{lem}

\begin{proof} If $F\in {\mathcal F}_{S, \vec{d}}^\RR$ then 
\begin{align*}
{}& (S^{-1/2}F)\cdot (S^{-1/2}F)^t= S^{-1/2}FF^tS^{-1/2}=S^{-1/2}SS^{-1/2}={\rm I}_k\\
{}& (S^{-1/2}F)^t S (S^{-1/2}F)=F^tF,
\end{align*}
the diagonal entries of the last matrix being $\|f_1\|^2, \ldots, \|f_n\|^2$. Thus the map indicated in the lemma is well defined.

One only needs to show that the inverse map $\tilde{\mathcal F} \to{\mathcal F}_{S, \vec{d}}^\RR$, $A\mapsto S^{1/2}A$ is well defined. 
If $A\in \tilde{\mathcal F}$ then
\begin{align*}
{}& (S^{1/2}A)\cdot (S^{1/2}A)^t= S^{1/2}AA^tS^{1/2}=S\\
{}& (S^{1/2}A)^t (S^{1/2}A)=A^tSA,
\end{align*}
and consequently the norms of the columns of $S^{1/2}A$  are $d_1, \ldots, d_n$, respectively. 
\end{proof} 
Our goal is to show that $\tilde{\mathcal F}$ is a path-connected subspace of ${\rm V}_k(\RR^n)$.

\begin{lem} \label{princbun} The map $p$ defined in Lemma \ref{gbundle} restricted to $\tilde{\mathcal F}$ induces a $G$-principal bundle between $\tilde{\mathcal F}$ and 
$\mu_\RR^{-1}(\vec{d})$.
\end{lem}

\begin{proof}
It is clear that $p(\tilde{\mathcal F})= \mu_\RR^{-1}(\vec{d})$. Also, $\tilde{\mathcal F}$ is a $G$-invariant subspace of ${\rm V}_k(\RR^n)$, because if $A\in \tilde{\mathcal F}$ and $R\in G$ then
$$(RA)^tS(RA)=A^tR^tSRA=A^tR^tRSA=A^t{\rm I}_kSA=A^tSA,$$
whose diagonal is $(d_1,\ldots, d_n)$. 
\end{proof}
 
 We are now in a position to prove the main result of the subsection.
 
 \noindent {\it Proof of Thm.~\ref{maininstr} (a)}. Above we defined the $G$-principal bundle $p:{\rm V}_k(\RR^n) \to {\rm F}\ell^\RR$ and its restriction 
 $\tilde{\mathcal F} \to \mu_\RR^{-1}(\vec{d})$. Fix a point ${\rm pt.}\in \mu_\RR^{-1}(\vec{d})$ and consider its fiber, which is homeomorphic to $G$. To each of the pairs $(\tilde{\mathcal F}, G)$ and $({\rm V}_k(\RR^n), G)$ corresponds a long exact homotopy sequence. They can be combined with the homomorphisms induced by the inclusion map
 $(\tilde{\mathcal F}, G)\hookrightarrow ({\rm V}_k(\RR^n), G)$ and determine the following ``ladder" diagram:

\begin{equation}\label{exacts}
\vcenter{\xymatrix{
&
\pi_1(\tilde{\mathcal F}, G) 
\ar[r]^{ \ \ \textcircled{\small{1}}}\ar[d]^{\textcircled{\small{2}}} &
\pi_0(G) 
\ar[r] \ar[d]^{\simeq} &
\pi_0(\tilde{\mathcal F})\ar[r] \ar[d]&  \pi_0(\tilde{\mathcal F}, G)\ar[d]
\\
&
\pi_1({\rm V}_k(\RR^n), G)
\ar[r]^{ \ \ \ \ \ \textcircled{\small{3}}} &
\pi_0(G)  \ar[r] & \pi_0({\rm V}_k(\RR^n)) \ar[r] &  \pi_0({\rm V}_k(\RR^n), G)
\\
}}
\end{equation}  
where each rectangle is a commutative diagram.
This follows as an application of a general result which can be found for instance in \cite[p.~76]{Gray}. 

{\it Claim.} The map $\textcircled{\small{2}}$ is surjective.

To prove this,  consider the following commutative diagram:

\begin{equation*}\label{exactseq}
\vcenter{\xymatrix{
&
(\tilde{\mathcal F}, G) 
\ar[r]^{}\ar[d]^{} &
(\mu_\RR^{-1}(\vec{d}), {\rm pt.}) 
\ar[d]^{} 
\\
&
({\rm V}_k(\RR^n), G)
\ar[r]^{} &
({\rm F}\ell^\RR, {\rm pt.})  
\\
}}\
\end{equation*}  

\noindent in which the horizontal arrows are the bundle projection maps  and the vertical arrows are the inclusions.
By functoriality, one obtains an obvious commutative diagram which involves the fundamental groups. In this new diagram, 
the horizontal arrows are isomorphisms, see \cite[Thm.~11.8]{Gray}. The left hand side vertical arrow is just the map 
$\textcircled{\small{2}}$ above and the other vertical arrow is surjective by the theorem hypothesis. This concludes the proof of the claim.

Since $k<n$, the Stiefel manifold ${\rm V}_k(\RR^n)$ is path-connected, that is,
$
\pi_0({\rm V}_k(\RR^n))
$ consists of only one element. 
Consequently, the map \textcircled{\small{3}} is surjective. Since \textcircled{\small{2}} is surjective, commutativity of diagram (\ref{exacts}) implies that \textcircled{\small{1}} is surjective as well.
By Theorem~11.8 of \cite{Gray}, applied to the principal bundle
$
p:\tilde{\mathcal F}\longrightarrow \mu_\RR^{-1}(\vec{d}),
$
the induced map
$
p_*:\pi_0(\tilde{\mathcal F},G)\longrightarrow \pi_0(\mu_\RR^{-1}(\vec{d}))
$
is a bijection.  By the hypothesis of the theorem and since ${\rm F}\ell^\RR$ is path-connected,  $\mu_\RR^{-1}(\vec{d})$ is path-connected as well.  Thus 
$
\pi_0(\tilde{\mathcal F},G)
$
consists of only one element.
Consider now the exact sequence appearing in the top row of (\ref{exacts}),
$$
\pi_1(\tilde{\mathcal F},G)
\stackrel{\textcircled{\small{1}}}{\longrightarrow}
\pi_0(G)
\longrightarrow
\pi_0(\tilde{\mathcal F})
\longrightarrow
\pi_0(\tilde{\mathcal F},G).
$$
Since \textcircled{\small{1}} is surjective, the map
$
\pi_0(G)\longrightarrow \pi_0(\tilde{\mathcal F})
$
is constant. On the other hand, since $\pi_0(\tilde{\mathcal F},G)$ is a singleton, the map
$
\pi_0(\tilde{\mathcal F})
\longrightarrow
\pi_0(\tilde{\mathcal F},G)
$
is constant. Exactness then implies that every element of $\pi_0(\tilde{\mathcal F})$ lies in the image of
$
\pi_0(G)\longrightarrow \pi_0(\tilde{\mathcal F}).
$
Since the latter map is constant, its image consists of a single element. Therefore $\pi_0(\tilde{\mathcal F})$ consists of a single element as well. Hence $\tilde{\mathcal F}$ is path-connected, and therefore, by Lemma \ref{lem:homeo},  ${\mathcal F}_{S, \vec{d}}^\RR$  is path-connected too. 
$\square$

\subsection{The case $\KK=\CC$}

With $n, r, k, k_1, \ldots, k_{r-1}$, $S$, and $\vec{d}$ like in Sect.~\ref{sec:mainres}, we now prove point (b) of Thm.~\ref{maininstr}.
As in the real case, one has a $G$-principal bundle
$
p:{\rm V}_k(\mathbb C^n)\to \Fl^{\mathbb C},
$ where this time
$G={\rm U}(k_1)\times\cdots\times {\rm U}(k_{r-1}),
$
and its restriction
$
\tilde{\mathcal F}\to  \mu_{\mathbb C}^{-1}(\vec d),
$
where $\tilde{\mathcal F}$ is a subspace of ${\rm V}_k(\CC^n)$ homeomorphic to ${\mathcal F}_{S, \vec{d}}^\CC$. 
These principal bundles fit into a commutative diagram whose vertical maps are the natural inclusions.
Consider the corresponding ladder of long exact homotopy sequences, which is the analogue of the diagram above, this time with the homotopy groups 
shifted up by one degree. Note that ${\rm V}_k(\mathbb C^n)$ is simply connected, because $k<n$. By hypothesis, the inclusion
$
\mu_{\mathbb C}^{-1}(\vec d)\hookrightarrow \Fl^{\mathbb C}
$
is 3-connected and, in particular, induces an isomorphism on  $\pi_2$.
It follows from the commutativity of the ladder and a careful analysis  of the  sequence in the top row that 
$
\pi_1(\tilde{\mathcal F})=0.
$

 In fact, this proof has already been used, in a slightly different context,  
in \cite[Sect.~4, Proof of Theorem 1.1]{Ma2}. 

\begin{rem} The proof above only uses that the inclusion
$
\mu_{\mathbb C}^{-1}(\vec d)\hookrightarrow \Fl^{\mathbb C}
$
induces a {\it surjection} on $\pi_2$, and not necessarily an isomorphism. 
This establishes more than what is stated in Thm.~\ref{maininstr} (b): the hypothesis that the inclusion 
$\mu_\CC^{-1}(\vec{d}) \to \Fl^\CC$  be 
3-connected can be relaxed to 
2-connected. 
\end{rem}

\section{Complex flag manifolds}\label{sect1}

The rest of the paper is devoted to the proof of Theorems \ref{thmconn1} and \ref{thmconn2}. 
The main instrument is of course Thm.~\ref{maininstr}. It remains to show that the assumptions in the two cases of that theorem are satisfied. We start with a  discussion about complex flag manifolds. 

Let $n\ge 2$, $r\ge 1$, and $1\le k_1, \ldots, k_r \le n$ be integers such that $k_1+ \cdots + k_r=n$. Let ${\rm Herm}_n$ be the space of all $n\times n$ Hermitian matrices. It is acted upon by the group ${\rm U}(n)$ using matrix conjugation. In this section, by
$\Fl^\CC$ we will denote the   ${\rm U}(n)$-orbit of the  diagonal matrix 
$${\rm Diag} (\underbrace{r, \dots, r}_{k_1},
\underbrace{r-1, \dots, r-1}_{k_2},
\dots,
\underbrace{1, \dots, 1}_{k_r}).$$
One naturally identifies $\Fl^\CC$ with the space of all Hermitian $n\times n$ matrices whose eigenvalues are $r, r-1, \ldots, 1$, each with multiplicity $k_1, \ldots, k_r$ respectively. Any such matrix / linear transformation is uniquely determined by its eigenspaces, say $V_1, \ldots, V_r$; that is, its restriction to each $V_j$ is just the multiplication by $r-j+1$, for all $1\le j\le r$. This allows us to further identify $\Fl^\CC$ with the space of all sequences $V_\bullet:=(V_1, \ldots, V_r)$ of $\CC$-linear subspaces of $\CC^n$, which are pairwise orthogonal and  have dimensions $k_1, \ldots, k_r$, respectively.

By the theorem of Schur and Horn, see \cite{Schur} and \cite{Ho}, the image of the matrix-diagonal map\footnote{In this section, we will use the same notation $\mu$ to refer to $\mu_\CC$ and also to the moment map of an arbitrary Hamiltonian torus action on a symplectic manifold: the actual meaning will be clear from the context.} $\mu: \Fl^\CC \to \RR^n$ is the permutahedron $\Delta_{k_1, \ldots, k_r}^n$, that is the convex hull of the subset of $\RR^n$  consisting of the vector   
$$a=(\underbrace{r, \dots, r}_{k_1},
\underbrace{r-1, \dots, r-1}_{k_2},
\dots,
\underbrace{1, \dots, 1}_{k_r})$$
along with all vectors obtained by permuting the entries of $a$.  
Let $x_0$ be  the barycenter of $\Delta_{k_1, \ldots, k_r}^n$. The latter is clearly invariant under the action of the symmetric group $S_n$ and thus $x_0$ must be fixed by $S_n$. But the only vectors in $\RR^n$ that are fixed under $S_n$ are those whose entries are all equal. Since the sum of the entries of all vectors in $\Delta_{k_1, \ldots, k_r}^n$ is the same\footnote{In general, if $x\in \RR^n$ then all vectors in the convex hull of $\{\sigma x \mid \sigma\in S_n\}$ have the sum of the entries constant.}, one deduces that
\begin{equation}\label{eq:xzero}x_0= \frac{\sum_{i=1}^r (r-i+1)k_i}{n}\,(1,\dots,1).
\end{equation}

\begin{rem} \label{rem:first}
It may look more natural to choose ${\rm F}\ell^\CC$ as the ${\rm U}(n)$-orbit of 
$${\rm Diag} (\underbrace{1, \dots, 1}_{k_1},
\underbrace{2, \dots, 2}_{k_2},
\dots,
\underbrace{r, \dots, r}_{k_r}).$$
This  looks more attractive, for example because the barycenter of the resulting polytope would then be
$$\frac{\sum_{i=1}^r i\, k_i}{n}\,(1,\dots,1)$$
and the computations in the next sections would get a simpler form.
In fact, there is no essential difference between the two presentations, since the orbit mentioned in this remark is the same as the orbit of 
$${\rm Diag} (\underbrace{r, \dots, r}_{k_r}, \dots
\underbrace{2, \dots, 2}_{k_{2}},
\underbrace{1, \dots, 1}_{k_1}).$$
However, it will be useful later on to choose  the ${\rm U}(n)$-orbit in such a way that for any matrix belonging to it, the eigenspaces $V_1, V_2, \ldots, V_r$ correspond to the eigenvalues $r, r-1, \ldots, 1$ and not $1, 2, \ldots, r$, respectively. This will be explained in more detail in Rem.~\ref{rem:second}.
\end{rem}

\subsection{The tangent space to the flag manifold}

Recall that ${\rm F}\ell^\CC$ is a submanifold of the  vector space ${\rm Herm}_n$. 
\begin{lem} \label{tangsp} For any $V_{\bullet}= (V_1, \ldots, V_r)\in {\rm F}\ell^\CC$, the tangent space  $T_{V_{\bullet}}{\rm F}\ell^\CC$ consists of all $\CC$-linear endomorphisms of the form
$$\sum_{1\le i <j \le r} T_{ij}+T_{ij}^*$$
where $T_{ij} : V_i \to V_j$ is an arbitrary $\CC$-linear map, for all $1\le i < j \le r$.  More concisely, one can identify 
\begin{equation}\label{eq:tangent}T_{V_\bullet} {\rm F}\ell^\CC =\bigoplus_{1\le i < j \le r} {\rm Hom}_\CC (V_i, V_j)\end{equation}
as real vector spaces. Under this identification, $T_{V_\bullet}\Fl^\CC$ inherits a natural stucture of complex vectors space. 
\end{lem}
To make the statement clear, some explanations are necessary. If $T_{ij} : V_i \to V_j$ is
$\CC$-linear, then $T_{ij}^*: V_j \to V_i$ is its adjoint, defined by
$$\langle T_{ij}^*(w), v\rangle = \langle w, T_{ij}(v)\rangle $$
for all $v\in V_i$ and $w\in V_j$. The map 
$\sum_{1\le i <j \le r} T_{ij}+T_{ij}^*$ is the $\CC$-linear endomorphism of $\CC^n=\bigoplus_{i=1}^r V_i$ described by
$$\left(\sum_{1\le i <j \le r} T_{ij}+T_{ij}^*\right)(v_1+ \cdots + v_r) =\sum_{1\le i <j \le r}
T_{ij}(v_i) + T_{ij}^*(v_j),$$
whenever $v_i \in V_i$, $1\le i \le r$. 

We will first prove the lemma in the special case when $r=2$, that is, for the ${\rm U}(n)$-conjugacy orbit in ${\rm Herm}_n$ of 
$$E= \begin{pmatrix}
2I_k& 0\\
0 &I_{n-k}
\end{pmatrix}, $$
which consists of all $AEA^{-1}$ with $A\in {\rm U}(n)$. 
The tangent space at $E$ consists of all matrices of the form $XE - EX$, where $X$ is in the Lie algebra of ${\rm U}(n)$, i.e.~it is $n\times n$ skew-Hermitian. 
Write $X$ in the block form, the pattern being the same as for $E$:
$$X=\begin{pmatrix}
S& -T^*\\
T &S'
\end{pmatrix},$$
where $S$ and $S'$ are skew-Hermitian. Compute:
$$XE-EX= 
\begin{pmatrix}
S& -T^*\\
T &S'
\end{pmatrix} \begin{pmatrix}
2I_k& 0\\
0 &I_{n-k}
\end{pmatrix} - \begin{pmatrix}
2I_k& 0\\
0&I_{n-k}
\end{pmatrix} \begin{pmatrix}
S& -T^*\\
T &S'
\end{pmatrix}=\begin{pmatrix}
0& T^*\\
T &0
\end{pmatrix}.
$$
This is exactly the form indicated in the lemma.

Let now $V_1$ be an arbitrary $k$-dimensional linear subspace  of $\CC^n$, which is of the form $A\CC^k$, for some $A\in {\rm U}(n)$. Also, $V_2:=A(V^\perp)$. As an element of the orbit ${\rm U}(n).E$, $(V_1, V_2)$ is equal to $A E A^{-1}$ and thus 
$$T_{(V_1, V_2)} ({\rm U}(n).E)=AT_{E}({\rm U}(n).E)A^{-1}.$$
The direct sum decomposition $\CC^n= \CC^k \oplus \CC^{n-k}$ induces 
$$\CC^n= A \CC^k \oplus A \CC^{n-k},$$
where $A\CC^k = V_1$ and $A\CC^{n-k}= V_2$. 
If $T: \CC^k \to \CC^{n-k}$ is linear then $ATA^{-1}$ maps $V_1$ to $A\CC^{n-k}$, which is $V_2$. 
Furthermore, the adjoint of $ATA^{-1}$ as a map from $V_2$ to $V_1$ is just $(A^{-1})^*T^*A^*$, that is,
$AT^*A^{-1}$. Consequently, the elements of $AT_{E}({\rm U}(n).E)A^{-1}$ are of the form
$$ATA^{-1} + AT^*A^{-1}=ATA^{-1}+(ATA^{-1})^*.$$

The case $r>2$ can be proved in a very similar way. One starts with
$$
E={\rm Diag}(
\underbrace{r,\ldots,r}_{k_1},
\ldots,
\underbrace{2,\ldots,2}_{k_{r-1}},
\underbrace{1,\ldots,1}_{k_r}
),
$$
writes an arbitrary skew-Hermitian matrix $X$ in corresponding block form, and computes $XE-EX$. The passage from $E$ to an arbitrary element of $\Fl^\CC$ is again straightforward.
\hfill $\square$


\subsection{The critical set of the squared-norm of the moment map according to Kirwan} \label{critkir}

Some general considerations concerning the topic mentioned in the title are needed at this point. Let $M$ be a compact connected symplectic manifold acted on by a torus $T$ and $\mu: M \to {\mathfrak t}^*$ a moment map. 
Fix an inner product on ${\mathfrak t}$, which is used to identify ${\mathfrak t}^*$ with ${\mathfrak t}$ and get 
$f: M \to \RR$, $f(x)=\|\mu(x)\|^2$, $x\in M$. To describe Crit($f$), consider  all $\beta \in {\mathfrak t}$ with the following property:
\begin{itemize}
	\item[(P)] If $\mu_\beta:=\langle \mu, \beta\rangle$, there exist a connected component  of ${\rm Crit}(\mu_\beta)$ such that $\beta$ lies in its $\mu$-image. 
	\end{itemize} 
	
	For any $\beta$ with the property above, we denote by $Z_\beta^i$ the  connected components of ${\rm Crit}(\mu_\beta)$ such that
	$\beta \in \mu(Z_\beta^i)$.
	 According to \cite[Lemma 3.12]{Ki}, the set ${\bf B}$ of all such vectors $\beta$  is finite and $${\rm Crit}(f)= \bigcup_{\beta \in {\bf B}, i} Z_\beta^i \cap \mu^{-1}(\beta).$$

\begin{rem} The definitions in Kirwan \cite{Ki} are slightly different, but the output is the same. Namely, she defines $Z_\beta$ as the union of all components of Crit($ \mu_\beta)$ on which $\mu_\beta$ takes the value $\|\beta\|^2$, see also \cite[Sect.~3.1]{BaHe}. After that she proves that Crit($f$) is the union of all $Z_\beta \cap \mu^{-1}(\beta)$ where $\beta$ ranges over a finite set slightly larger than the one considered here; see  \cite[Lemma 3.15]{Ki}). This is equivalent to the formulation above. Indeed, in the intersection $Z_\beta\cap\mu^{-1}(\beta)$, only those connected components of ${\rm Crit}(\mu_\beta)$ whose $\mu$-image contains $\beta$ can contribute. Conversely, if a connected component $Z_\beta^i$ of ${\rm Crit}(\mu_\beta)$ satisfies
$
\beta\in\mu(Z_\beta^i),
$
then $\mu_\beta$ is constant on $Z_\beta^i$, and its value is necessarily
$
\mu_\beta(\beta)=\langle\beta,\beta\rangle=\|\beta\|^2.
$
Thus $Z_\beta^i$ is one of the components included in Kirwan's $Z_\beta$.
\end{rem}


\subsection{The critical set of $\|\mu -x_0\|^2$} \label{sect:critmu}

The flag manifold $\Fl^\CC$ defined at the beginning of this section can be canonically equipped with a symplectic form. To this end one needs to identify 
first the space ${\rm Herm}_n$ with ${\frak u}(n)$, the Lie algebra of ${\rm U}(n)$, which consists of all $n\times n$ 
skew-Hermitian matrices, by means of the map $X\mapsto \sqrt{-1}X$, for all $X\in {\rm Herm}_n$. After that, ${\frak u}(n)$ needs to be identified with its real dual
${\frak u}(n)^*$, in terms of the map 
$$Y \mapsto (Z \mapsto {\rm Tr}(ZY)),$$
for all $Y, Z\in {\frak u}(n)$.  These make $\Fl^\CC$ into a coadjoint orbit of ${\rm U}(n)$, which can be equipped with the Kirillov-Kostant-Souriau symplectic form.
 Let now $T$ be the set of all diagonal matrices in ${\rm U}(n)$, which is an $n$-torus. 
Its induced action on $\Fl^\CC$ is Hamiltonian and a moment map $\mu: \Fl^\CC \to \RR^n$ is just 
$\mu_\CC$,  which was previously defined.  
(Here the Lie algebra of $T$ is ${\mathfrak t}\simeq {\mathbb R}^n$, which is identified with ${\mathfrak t}^*$ by means of the canonical inner product.) Let us now take $\beta =(\beta_1, \ldots, \beta_n) \in \RR^n$ an  arbitrary vector and determine the critical set of $\mu_\beta$ is.

Consider the partition
\begin{equation} \label{1dotsn}\{1, \ldots,n\} = A_1 \sqcup \cdots \sqcup A_m,\end{equation}
where $A_1, \ldots, A_m$ are the equivalence classes associated to $$i\sim j 
 \ \stackrel{{\rm def.}}{\Leftrightarrow} \  \beta_i=\beta_j, \ 1\le i,j\le n.$$  
 Denote $n_s:=|A_s|$ and $\CC^{A_s}:={\rm Span}_\CC \{e_i \mid i\in A_s\}$, for all $1\le s \le m$.
 
By \cite[Sec.~3.7]{Ki}, ${\rm Crit}(\mu_\beta)$ consists of all flags $V_\bullet=(V_i)_{1\le i \le r}\in \Fl^\CC$ which are fixed by all $\exp(t\beta)$, $t\in \RR$.
This can be expressed by saying that $\beta$, regarded as a diagonal linear transformation, leaves each $V_i$ invariant, or, equivalently,
$$
V_i=(V_i\cap \CC^{A_1}) \oplus \cdots \oplus (V_i \cap \CC^{A_m}),
$$
for all $1\le i \le r$. This observation allows us to describe the critical set as follows.
 



Take integers $k_{i,s} \ge 0$ indexed by
$1 \le i \le r$  and $1 \le s \le m$ 
such that
\begin{equation}
\sum_{s=1}^m k_{i,s} = k_i \,\text{ for all } i \ {\rm and} \ 
\sum_{i=1}^r k_{i,s} = n_s \,\text{ for all } s.
\label{eq:matching}
\end{equation}

For any such  matrix $\overline{k} = (k_{i,s})$ we look at orthogonal  decompositions of the form
\begin{equation}\label{cs}
\CC^{A_s} = V_{1,s} \oplus V_{2,s} \oplus \cdots \oplus V_{r,s}, \ {\rm where} \  V_{i,s}\subset \CC^{A_s} \ {\rm is \ a \ } \CC{\rm -linear  \
subspace}, \ \dim V_{i,s} = k_{i,s},
\end{equation} for  all   $1\le i \le r$ and $1\le s \le m.$

The connected components of  ${\rm Crit}(\mu_\beta)$ are precisely the submanifolds
\begin{equation}
\mathcal{C}(\overline{k}) \
:= \prod_{s=1}^m {\rm F}\ell^\CC_{k_{1,s}, k_{2,s}, \ldots, k_{r,s}}(\CC^{A_s}),
\label{eq:criticalcomponent}
\end{equation}
corresponding to the matrices $\overline{k}=(k_{i,s})$ which satisfy the conditions in \eqref{eq:matching}.
Here $ {\rm F}\ell^\CC_{k_{1,s}, k_{2,s}, \ldots, k_{r,s}}(\CC^{A_s})$ consists of all $V_{1,s}, V_{2,s}, \ldots, V_{r,s}$, which are $\CC$-linear subspaces of $\CC^{A_s}$, are pairwise orthogonal, and satisfy Eq.~(\ref{cs}). 
The embedding of the direct product above into $\Fl^\CC={\rm F}\ell^\CC_{k_{1}, k_{2}, \ldots, k_{r}}(\CC^{n})$ is obvious: to any
collection $V_{1,s}, V_{2,s}, \ldots, V_{r,s}$, with $1\le s \le m$ corresponds
$$\bigoplus_{s=1}^m V_{1,s}, \bigoplus_{s=1}^m V_{2,s}, \ldots , \bigoplus_{s=1}^m V_{r,s}.$$

Consider now the function $f_0 : {\rm F}\ell^\CC \to {\mathbb R}$, $f_0(V_\bullet)=\|\mu(V_\bullet)\|^2$, $V_\bullet \in {\rm F}\ell^\CC$. To describe its critical set, we will use the procedure described in Sect.~\ref{critkir}.  We are searching for $\beta \in \RR^n$ satisfying the requirement (P). Having in mind the presentation of Crit($\mu_\beta$), the 
critical components $Z_\beta^i$ are of the form $\mathcal{C}(\overline{k})$, see Eq.~(\ref{eq:criticalcomponent}), 
where the integers $k_{i,s}$ satisfy the conditions in (\ref{eq:matching}). 
Note that the  $\mu$-image of any such component is
$$\mu(\mathcal{C}(\overline{k})) = \prod_{s=1}^m \Delta^{A_s}_{k_{1,s}, k_{2,s}, \ldots, k_{r,s}},$$
where $\Delta^{A_s}_{k_{1,s}, k_{2,s}, \ldots, k_{r,s}} \subset \RR^{A_s}$ is the convex hull of the vector
$$(\underbrace{r, \dots, r}_{k_{1,s}},
\underbrace{r-1, \dots, r-1}_{k_{2,s}},
\dots,
\underbrace{1, \dots, 1}_{k_{r,s}})$$
and the vectors obtained from it by permuting its entries. 
The condition $\beta \in \mu({\mathcal C}(\overline{k}))$ determines $\beta$ uniquely, being equivalent to the fact that the common value of the entries of $\beta$ of index in $A_s$ is equal to
$$\frac{\sum_{i=1}^r(r-i+1) k_{i, s}}{n_s},$$
for all $1\le s\le m$. 
This is because the sum of the entries of all vectors in $\Delta^{A_s}_{k_{1,s}, k_{2,s}, \ldots, k_{r,s}} $ is the same. 
We have shown that  $\beta$ satisfies the requirement (P) if and only if 
$$\beta=\sum_{s=1}^m\left(\frac{\sum_{i=1}^r(r-i+1) k_{i, s}}{n_s}\sum_{j\in A_s} e_j \right).$$
For such a $\beta$, there are in general several components $Z_\beta^i$ which satisfy the condition (P), namely
$$Z_\beta^i= \prod_{s=1}^m {\rm F}\ell^\CC_{\ell_{1,s}, \ell_{2,s}, \ldots, \ell_{r,s}}(\CC^{A_s}),$$
where  $(\ell_{i,s})_{1\le i \le r, 1\le s\le m}$ satisfy the conditions in (\ref{eq:matching}) and 
$${\sum_{i=1}^r (r-i+1)\ell_{i, s}}={\sum_{i=1}^r(r-i+1)k_{i, s}}, \ {\rm for \ all} \ 1\le s \le m.$$

Let us now make the following notations. To any partition ${\mathcal A}=\{A_1, \ldots, A_m\}$ and any matrix
$\overline{k}=(k_{i,s})$  which satisfies  the conditions in (\ref{eq:matching}) one associates 
$$\beta_{{\mathcal A},\overline{k}} := \sum_{s=1}^m\left(\frac{\sum_{i=1}^r(r-i+1) k_{i, s}}{n_s}\sum_{j\in A_s} e_j \right),$$
 and
 \begin{equation}\label{zak}Z_{{\A,\overline{k}}}:=  \prod_{s=1}^m {\rm F}\ell^\CC_{k_{1,s}, k_{2,s}, \ldots, k_{r,s}}(\CC^{A_s}).\end{equation}
 Also, denote by $[\overline{k}]$ the equivalence class of the matrix $\overline{k}$ relative to the equivalence relation given by
\begin{equation}\label{sime}{\overline{\ell}} \sim \overline{\ell}' \ \stackrel{{\rm def.}}{\Leftrightarrow}  \ \beta_{\A,\overline{\ell}}=\beta_{\A, \overline{\ell}'}\end{equation}
 for any $\overline{\ell}=(\ell_{i,s})_{1\le i \le r, 1\le s\le m}$, $\overline{\ell}'=(\ell'_{i,s})_{1\le i \le r, 1\le s\le m}$ which satisfy the conditions in (\ref{eq:matching}). Finally
\begin{equation}\label{za}Z_{{\A,[\overline{k}]}}:=\bigcup_{\overline{\ell} \sim \overline{k}} Z_{{\A,\overline{\ell}}}.\end{equation}
The discussion above shows that the critical set of $f_0$ is the (disjoint) union 
$${\rm Crit}(f_0)= \bigsqcup (Z_{{\A,[\overline{k}]}} \cap \mu^{-1}(\beta_{\A,\overline{k}})),$$
the indexing set being the one of all pairs $(\A, [\vk])$, where ${\mathcal A}=\{A_1, \ldots, A_m\}$ is a partition of 
 $\{1, \ldots,n\}$ and $\overline{k}=(k_{i,s})_{1\le i \le r, 1\le s\le m}$ a matrix which satisfies  the conditions in (\ref{eq:matching}).
 
With the hypotheses of Thm.~\ref{maininstr} (b) in mind, we are actually interested in $\mu^{-1}(x_0)$, where $x_0$ is given by Eq.~(\ref{eq:xzero}). Set
 $$u:=\frac{\sum_{i=1}^r (r-i+1)k_i}{n},$$
 so that $x_0=(u, \ldots, u).$ Consider
  $\tilde{\mu} : {\rm F}\ell^\CC \to {\mathbb R}^n$ and  $f : {\rm F}\ell^\CC \to \RR$ 
  $$\tilde{\mu}(V_\bullet)=\mu(V_\bullet)-x_0,\quad f(V_\bullet) = \|\tilde{\mu}(V_\bullet)\|^2 ,$$ for all $V_\bullet \in {\rm F}\ell^\CC$.   Just as in the first part of this subsection, we are searching for all $\bet \in \RR^n$  which satisfy the requirement (P), this time relative to $\tilde{\mu}$. To this end observe that $\tilde{\mu}_{\bet}(V_\bullet)=\mu_{\bet}(V_\bullet)-\langle x_0, \bet\rangle$, thus ${\rm Crit}(\tilde{\mu}_{\bet})={\rm Crit}({\mu}_{\bet})$.
  Again, $\tilde{\beta}$ induces the partition given by Eq.~(\ref{1dotsn}) and the components of Crit($\tilde{\mu}_{\tilde{\beta}}$) are described by Eq.~(\ref{eq:criticalcomponent}). Say that 
$$Z_{\bet}^i=\prod_{s=1}^m {\rm F}\ell^\CC_{k_{1,s}, k_{2,s}, \ldots, k_{r,s}}(\CC^{A_s}).$$
Then 
$$\tilde{\mu}(Z_{\bet}^i)={\mu}(Z_{\bet}^i)-x_0$$
which consists of all
$$x=\left(x_1-u, \ldots,  x_n-u\right)$$
 with $(x_1, \ldots, x_n)\in \prod_{s=1}^m \Delta^{A_s}_{k_{1,s}, k_{2,s}, \ldots, k_{r,s}}$. 
 For $\tilde{\beta}$ to be in $\tilde{\mu}(Z_{\tilde{\beta}}^i)$ it is necessary and sufficient that $\tilde{\beta}+x_0$ has the common value of the entries of index in $A_s$  equal to  
 $$\frac{\sum_{i=1}^r(r-i+1) k_{i, s}}{n_s},$$
for all $1\le s\le m$.

We conclude:

\begin{lem}\label{critset}
 The critical set of   $f  : {\rm F}\ell^\CC \to \RR$, $f(V_\bullet) = \|\mu(V_\bullet)-x_0\|^2$, $V_\bullet \in \Fl^\CC$,   is labeled by:
 \begin{itemize}
 \item integers
 $1\le m \le n$,
 \item collections ${\mathcal A}=\{A_1, \ldots, A_m\}$ such that
 $$\{1, \ldots,n\} = A_1 \sqcup \cdots \sqcup A_m,$$
\item integers $k_{i,s}\ge 0$, with $1\le i \le r$ and $1\le s\le m$  which satisfy the conditions in (\ref{eq:matching}) 
 and such that the numbers 
 $$\frac{\sum_{i=1}^r(r-i+1) k_{i, s}}{n_s}, \ 1\le s \le m$$
  are pairwise distinct.
 \end{itemize} 
 For any such choice, denote
\begin{align}\label{betatil} {}& \tilde{\beta}_{{\mathcal A},\overline{k}} = \sum_{s=1}^m\left(\frac{\sum_{i=1}^r(r-i+1) k_{i, s}}{n_s}\sum_{j\in A_s} e_j \right)-x_0\\ \nonumber
{}& = \sum_{s=1}^m\left(\left(\frac{\sum_{i=1}^r(r-i+1)k_{i, s}}{n_s}-\frac{\sum_{i=1}^r {(r-i+1)k_i}}{n}\right)\sum_{j\in A_s} e_j \right).\end{align}
Then ${\rm Crit}(f)$ is equal to the following (disjoint) union
$${\rm Crit}(f)= \bigsqcup (Z_{{\A,[\overline{k}]}} \cap \tilde{\mu}^{-1}(\tilde{\beta}_{\A,\overline{k}})),$$
the indexing set being the one described above. Here $Z_{{\A,[\overline{k}]}}$ is defined by Equations (\ref{za}) and (\ref{zak}).
 \end{lem} 

We denote by ${\bf B}$ the set of  all pairs $({\A, [\vk]})$ described in the lemma. 

\begin{rem} It is crucial for the above simple description of the vectors $\tilde{\beta}$  that $x_0\in \RR^n$ have all its coordinates equal. For a general vector $x_0$, the condition
$\tilde\beta\in\tilde\mu(Z_{\tilde\beta}^i)$
leads to a more complicated system of conditions. We expect that the general situation can still be treated, but the corresponding analysis appears substantially more involved. In particular, such an extension would potentially allow the results of this section to be applied to frame spaces whose column vectors do not all have the same norm.
\end{rem}

\begin{rem} \label{rem:more}
In general, the equivalence class $[\vk]$ relative to the equivalence relation (\ref{sime}) contains more than just $\vk$.  
For example, take $n=6$, $r=3$, $k_1=k_2=k_3=2$, $m=2$, $A_1=\{1,2,3\}$, $A_2=\{4,5,6\}$, and

$$
\overline{k}=
\begin{pmatrix}
0&2\\
2&0\\
1&1
\end{pmatrix},
\qquad
\overline{\ell}=
\begin{pmatrix}
1&1\\
0&2\\
2&0
\end{pmatrix}.
$$

Both matrices satisfy the conditions in (\ref{eq:matching}) and 

$$\frac{3k_{1,1}+2k_{2,1}+k_{3,1}}{3}=\frac{3\ell_{1,1}+2\ell_{2,1}+\ell_{3,1}}{3}
=\frac{5}{3},
$$

while

$$\frac{3k_{1,2}+2k_{2,2}+k_{3,2}}{3}=\frac{3\ell_{1,2}+2\ell_{2,2}+\ell_{3,2}}{3}
=\frac{7}{3}.$$

Consequently,  $$\beta_{\A,\vk}=\beta_{\A, \overline{\ell}}
=\frac{5}{3}(e_1+e_2+e_3)
+
\frac{7}{3}(e_4+e_5+e_6).
$$


\end{rem}

\begin{rem}\label{betaz} The special case $m=1$ will be important later on. In this case,
$$A_1=\{1, \ldots, n\}, \ k_{i,1}=k_i, \ {\rm for \ all} \  1 \le i \le r, \  \ {\rm and} \ n_1=n.$$
For future reference, denote the resulting partition by 
$\A_0$ and the resulting matrix by $\vk_0$. 
Clearly $\tilde{\beta}_{\A_0, \vk_0}=(0, \ldots, 0)$, 
$Z_{\A_0, \vk_0}$ is the whole ${\rm F}\ell^\CC$, and $[\overline{k}_0]=\{\overline{k}_0\}$. 
\end{rem}

\subsection{The Kirwan stratification of ${\rm F}\ell^\CC$} \label{complex}

 According to Kirwan \cite[Thm.~4.16]{Ki}, there exists a stratification of $\Fl^\CC$, indexed by ${\mathbf B}$. In order to describe it, we first need to consider 
 a pair $(\A, \vk)$ of the type indicated in Lemma \ref{critset} and compute the Morse index of $\tilde{\mu}_{\tilde{\beta}_{\A, \vk}}$ along any of the critical components 
 $Z_{\A, \overline{\ell}}$, with $\overline{\ell}\sim\overline{k}$. To begin with, let us focus on the index of the aforementioned function along $Z_{\A, \overline{k}}$. Recall that the latter is describe by Eq.~(\ref{zak}).
 Denote by $Y_{\be}$ the stable manifold of $\tilde{\mu}_{\tilde{\beta}_{\A, \vk }}$ along $Z_{\be}$: we actually need its codimension in $\Fl^\CC$. 
 
 The computation relies on the description of the tangent space $T_{V_\bullet} Y_{\be}$, where $V_\bullet \in Z_{\be}$, based on general considerations that can be found  in \cite{Ki}. 
One needs to start with the torus $T_{\be}\subset T$ generated by $\bet_{\be}$. Concretely, this  is the following topological closure 
in $T$:
$$T_{\be}=\overline{\{\exp ( t\bet_{\be}) \mid t \in \RR\}}.$$
Obviously,
$$T_{\be} \subset \Delta_{A_1}(S^1) \times \Delta_{A_2}(S^1) \times \cdots \times \Delta_{A_m}(S^1),$$
where 
$$\Delta_{A_s}(S^1):=\{z \sum_{j\in A_s} e_j \mid z \in S^1\}.$$
One can see from this that $T_{\be}$ acts trivially on $Z_{\be}$. 
  We would like to describe the action of $T_{\be}$ on $T_{V_\bullet} {\rm F}\ell^\CC$.  
  We first look at $V_\bullet$, which is in $Z_{\be}$,  and thus has the form
  $$V_1=\bigoplus_{s=1}^m V_{1,s},  \ V_2=\bigoplus_{s=1}^m V_{2,s},  \ \ldots ,  \ V_r=\bigoplus_{s=1}^m V_{r,s},$$
  where
    $V_{1,s}, V_{2,s}, \ldots, V_{r,s}$ are $\CC$-linear subspaces of $\CC^{A_s}$, pairwise orthogonal, of dimensions $k_{1,s}, \ldots, k_{r,s},$ respectively.
    
   By Lemma \ref{tangsp}, 
   $$T_{V_\bullet}{\rm F}\ell^\CC = \bigoplus_{1\le i < j \le r} \bigoplus_{1\le s , t \le m}  {\rm Hom}_\CC (V_{i,s}, V_{j, t}).$$
The vector $\bet_{\be}$ induces an infinitesimal generator, which is defined as:
$$\bet_{\be} . \varphi = \frac{d}{d\tau}\Big|_{\tau=0}\left(\exp (\tau \bet_{\be}) .\varphi\right),$$
for all $\varphi \in T_{V_\bullet}\Fl^\CC$ ($\varphi$ is actually a sum of homomorphisms).

\begin{lem}
For any $1 \le i <j \le r$ and $1 \le s, t \le m$, $\bet_{\be}$ leaves ${\rm Hom}_\CC (V_{i,s}, V_{ j, t})$ invariant. Moreover, for any
$\varphi\in  {\rm Hom}_\CC (V_{i,s}, V_{j, t})$,
$$\bet_{\be}.\varphi = \sqrt{-1}\left(\frac{\sum_{i=1}^r(r-i+1)k_{i,t}}{n_t}-\frac{\sum_{i=1}^r(r-i+1)k_{i,s}}{n_s}\right)\varphi.$$
\end{lem}

\begin{proof}
For any $1\le s \le m$, $v\in \CC^{A_s}$, and $\tau \in \RR$,
$$\exp(\tau\bet_{\be}).v=e^{\sqrt{-1}\tau\left(\frac{\sum_{i=1}^r(r-i+1)k_{i,s}}{n_s}-\frac{\sum_{i=1}^r(r-i+1)k_i}{n}\right)}v.$$
Consequently, for any $\varphi\in {\rm Hom}_\CC (V_{i,s}, V_{j,t})$ and any $v\in V_{i,s}$,
$$(\exp(\tau \bet_{\be}).\varphi)(v)=e^{\sqrt{-1}\tau\left(\frac{\sum_{i=1}^r(r-i+1)k_{i,t}}{n_t}-\frac{\sum_{i=1}^r(r-i+1)k_i}{n}\right)}\varphi(e^{-\sqrt{-1}\tau\left(\frac{\sum_{i=1}^r(r-i+1)k_{i,s}}{n_s}-\frac{\sum_{i=1}^r(r-i+1)k_i}{n}\right)}v).$$
Then $\bet_{\be} .\varphi$ is the derivative relative to $\tau$ at $\tau=0$ of the expression above. 
This gives the desired result.
\end{proof}

Let us denote
$$\kappa_s
:=
\frac{\sum_{i=1}^r (r-i+1)k_{i,s}}{n_s},$$
for all $1\le s\le m$. 
The lemma above implies immediately that $\bet_{\be}$ is identically zero on the subspace 
$$T_{V_\bullet}Z_{\be}= \bigoplus_{s=1}^m \bigoplus_{1\le i <j \le r} {\rm Hom}_\CC( V_{i,s}, V_{j,s})$$  
of $T_{V_\bullet}{\rm F}\ell^\CC$. 
By the description of the tangent space to the stable manifold in \cite[p.~52]{Ki}, $T_{V_\bullet} Y_{\be}$ is the direct sum of all ${\rm Hom}_\CC(V_{i,s}, V_{j,t})$ with the property that 
$1\le i<j\le r$ and $1\le s, t \le m$ such that
$$\kappa_t\ge \kappa_s.$$
Thus one obtains a direct complement of $T_{V_\bullet}Y_{\be}$ in $T_{V_\bullet} {\rm F}\ell^\CC$ as the direct sum of all ${\rm Hom}_\CC (V_{i,s}, V_{j,t})$, where $1\le i < j \le r$ and $1\le s, t\le m$ such that
 $\kappa_s>\kappa_t.$ 
 Since the (real) dimension of ${\rm Hom}_\CC (V_{i,s}, V_{j,t})$ is equal to $2k_{i,s} \, k_{j,t}$, we deduce that:
 
 \begin{prop}\label{sum} If $m \neq 1$ then the codimension of $Y_{\be}$ in $\Fl^\CC$ is equal to
\begin{equation}\label{condi}2\sum k_{i,s} \, k_{j,t}\end{equation}
 where $1\le i <j \le r$ and $1\le s,t\le m$ such that $\kappa_s> \kappa_t.$
 \end{prop}

 As already mentioned, the number indicated in the lemma is also equal to the index of $\tilde{\mu}_{\tilde{\beta}_{\A, \vk}}$ along  $Z_{\A, \overline{k}}$. Notice that this number depends exclusively on $\tilde{\beta}_{\A, \vk}$. In particular, if $\overline{\ell} \sim \vk$ relative to the equivalence relation given by Eq.~(\ref{sime}), then $ \tilde{\beta}_{\A, \vk}=\tilde{\beta}_{\A, \overline{\ell}}$ and thus the index of $\tilde{\mu}_{\tilde{\beta}_{\A, \vk}}$   
 along $Z_{\A, \overline{\ell}}$ is equal to the index along $Z_{\A, \overline{k}}$. That is, all relevant components of ${\rm Crit}(\tilde{\mu}_{\tilde{\beta}_{\A, \vk}})$ have the same Morse index. This observation enables us to apply   \cite[Thm.~4.16]{Ki}, see also \cite[Sect.~3.1]{BaHe}.
Concretely,  there exists a stratification
\begin{equation} \label{stratif} {\rm F}\ell^\CC = \bigsqcup_{(\A, [\vk]) \in {\bf B}}S_{\A, [\vk]},\end{equation} 
 with the following properties:
 \begin{itemize}
 \item For all $(\A, [\vk]) \in {\bf B}$, $S_{\A, [\vk]}$ is a submanifold of ${\rm F} \ell^\CC$ whose dimension is the same as the dimension of
 $Y_{\be}$.
 \item For all $(\A, [\vk]) \in {\bf B}$, $S_{\A, [\vk]}$ contains $Z_{{\A,[\overline{k}]}} \cap \tilde{\mu}^{-1}(\tilde{\beta}_{\A,\overline{k}})$ as a deformation retract (for this we refer to Lerman's work \cite{Le}).
 \item If {\bf B} is equipped with the strict partial order $\prec$ defined by
 $$ (\A, [\vk]) \prec (\A', [\overline{k'}]) \stackrel{{\rm def.}}{\Leftrightarrow} \|\tilde{\beta}_{\A', \overline{k'}}\| < \|\tilde{\beta}_{\A, \vk}\|$$ then 
 $$\overline{S_{\A,[\vk]}} \subset S_{\A,[\vk]}\cup \bigcup_{(\A', [\overline{k'}]) \prec (\A, [\vk]) }S_{\A', [\overline{k'}]},$$ where the left hand side is the topological closure of 
 $S_{\A,[\vk]}$ in
 $\Fl^\CC$. 
 \item There exists exactly one stratum of  codimension equal to 0, namely $S_{\A_0, [\vk_0]}$, see Rem.~\ref{betaz}.  It contains 
 $\tilde{\mu}^{-1}(0)=\mu^{-1}(x_0)$ as a deformation retract. 
 \end{itemize}

We are able to find a lower bound for the codimension of the strata above:

\begin{prop} \label{codimc}
Assume that $k_i \ge 2$ for all $1\le i \le r$. Then for any $(\A, [\vk] )\in {\bf B}$ which is different from $(\A_0, [\vk_0])$, the codimension of $S_{\A, [\vk]}$ in $\Fl^\CC$ is at least equal to 4. 
\end{prop}


\begin{proof} Since $S_{\A, [\vk]}$ and $Y_{\be}$ have the same dimension, we can use the formula for the codimension of the latter indicated in Prop.~\ref{sum}.
 Denote $$\Sigma:= \sum k_{i,s} k_{j,t}$$ the sum given by Eq.~(\ref{condi}). We need to show that  $\Sigma \ge 2$. This will be proved gradually, as follows.

{\it Step 1.} {\it For any pair $(s, t)$ with $\kappa_s > \kappa_t$ there exists a non-zero contribution $k_{i,s} \, k_{j,t}$ to $\Sigma$ (with $1\le i<j\le r$).}

From $\kappa_s > \kappa_t$, multiplying both sides by $n_s n_t > 0$ one obtains:
\[
  n_t \sum_{i=1}^r (r-i+1)\,k_{i,s} > n_s \sum_{i=1}^r (r-i+1)\,k_{i,t}.
\]
Substituting $n_s = \sum_{j=1}^r k_{j,s}$ and $n_t = \sum_{j=1}^r k_{j,t}$ and expanding,
\[
  \sum_{1\le i,j\le r} (r-i+1)\,k_{i,s}\,k_{j,t} >\sum_{1\le i,j\le r}(r-i+1)\,k_{i,t}\,k_{j,s}.
\]
Relabelling $i \leftrightarrow j$ in the right-hand side gives:
\[
  \sum_{1\le i,j\le r}(r-i+1)\,k_{i,s}\,k_{j,t} >\sum_{1\le i,j\le r}(r-j+1)\,k_{j,t}\,k_{i,s},
\]
and thus
\[
  \sum_{1\le i,j\le r}(j - i)\,k_{i,s}\,k_{j,t} >0.
\]
Separating the sum according to the sign of $j - i$:
\begin{equation}
  \sum_{i < j}(j-i)\bigl(k_{i,s}\,k_{j,t} - k_{j,s}\,k_{i,t}\bigr) > 0. \label{star}
\end{equation}
Consequently, there exists $i,j$ with $1\le i<j\le r$ such that
$$k_{i,s}\,k_{j,t} - k_{j,s}\,k_{i,t}>0 \Rightarrow k_{i,s}\,k_{j,t} > k_{j,s}\,k_{i,t} \Rightarrow k_{i,s}\,k_{j,t} >0.$$

{\it Step 2.} {\it The case $m>2$.}  By Lemma \ref{critset}, the numbers $\kappa_1, \ldots, \kappa_m$ are pairwise distinct. Since $m> 2$, there exist at least two distinct pairs $(s,t)$ such that $\kappa_s > \kappa_t$. By Step 1, each of these pairs induces a non-zero contribution to
$\Sigma$, and therefore $\Sigma \ge 2$.

{\it Step 3.} {\it The case $m=2$.} 
Without loss of generality we assume that $\kappa_1>\kappa_2$, so that
$$\Sigma=\sum_{1\le i<j\le r} k_{i,1}k_{j,2}.$$
Since $\kappa_1>\kappa_2$, using again Eq.~(\ref{star}),
\begin{equation}\label{key}
  \sum_{1 \le i < j \le r} (j - i)\,k_{i,1}\,k_{j,2}
  \;>\;
  \sum_{1 \le i < j \le r} (j - i)\,k_{j,1}\,k_{i,2}. 
  \end{equation}
Suppose $\Sigma = 1$. Then there exists a unique pair $(i_0, j_0)$ with $i_0 < j_0$ such that
$k_{i_0,1}\,k_{j_0,2} = 1$, and all other terms $k_{i,1}\,k_{j,2}$ (for $i < j$,
$(i,j)\ne(i_0,j_0)$) vanish. In particular $k_{i_0,1} = k_{j_0,2} = 1$, so the
left-hand side of Eq.~(\ref{key}) equals 
\[
  \sum_{1 \le i < j \le r} (j - i)\,k_{i,1}\,k_{j,2} = j_0 - i_0.
\]
On the other hand, since $k_{i_0} \ge 2$ and $k_{i_0,1} = 1$ we have $k_{i_0,2} \ge 1$;
similarly, since $k_{j_0} \ge 2$ and $k_{j_0,2} = 1$ we have $k_{j_0,1} \ge 1$.
Therefore the right-hand side of Eq.~(\ref{key}) satisfies
\[
  \sum_{1 \le i < j \le r} (j-i)\,k_{j,1}\,k_{i,2}
  \ge (j_0 - i_0)\,k_{j_0,1}\,k_{i_0,2}
  \ge (j_0 - i_0) \cdot 1 \cdot 1
  =j_0 - i_0.
\]
This contradicts the strict inequality \eqref{key}. 
\end{proof}
 
 \begin{rem}
 The assumption that $k_i\ge 2$ for all $1\le i \le r$ is essential in the proof above (see Step 3). \end{rem}
 
 \begin{rem} \label{rem:second} As already pointed out in Rem.~\ref{rem:first}, it is also essential for the proof above that the 
 ${\rm U}(n)$-orbit $\Fl^\CC$ we deal with consists of matrices with eigenspaces $V_1, V_2, \ldots, V_r$ whose eigenvalues are  
 $r, r-1, \ldots, 1$. If we change the order of the eigenvalues, the conclusion of Prop.~\ref{codimc} can fail to be true. 
More concretely, let us briefly consider the case when  ${\rm F}\ell^\CC$ is the orbit of  $${\rm Diag} (\underbrace{1, \dots, 1}_{k_1}, \underbrace{2, \dots, 2}_{k_2}, \dots,\underbrace{r, \dots, r}_{k_r}).$$   The barycenter would be then $$\frac{\sum_{i=1}^r i\, k_i}{n}\,(1,\dots,1).$$ By setting $$\kappa'_s := \frac{\sum_{i=1}^r i\,k_{i,s}}{n_s},$$ the challenge we would face is whether the sum $$\sum k_{i,s} k_{j,t}$$ where $1\le i <j \le r$ and $1\le s,t\le m$ with $\kappa'_s>\kappa'_t$ is different from 1.  The answer is in general negative, i.e., the sum can be equal to 1.
 Here is an example. 
 
 We start with the  ${\rm U}(5)$-orbit of
 $${\rm Diag} (1,1,2,2,2),$$ which we denote\footnote{Our orbit only gives  a certain presentation of the Grassmannian.}  ${\rm Gr}_2(\CC^5)$.  
 We have $n=5$, $k_1=2$, and $k_2=3$.  The barycenter of the image of the matrix-diagonal map is
 $$x_0=\frac{8}{5}(1,1,1,1,1).$$
 The critical set of $f$ can be described in the same way as in Lemma \ref{critset}. 
 Consider the connected component of ${\rm Crit}(f)$ attached to:
 \begin{itemize}
 \item $A_1=\{1,3\}$ and $A_2=\{2,4,5\}$ (thus $m=2$, $n_1=2$, and $n_2=3$),
 \item $k_{11}=1$, $k_{12}=1$, $k_{21}=1$, and $k_{22}=2$.
 \end{itemize}
 Thus 
 \begin{align*}
 {}& \kappa'_1=\frac{1\cdot 1 + 2\cdot 1}{2}=\frac{3}{2},\\
 {}& \kappa'_2=\frac{1\cdot 1 + 2 \cdot 2}{3}=\frac{5}{3}.
 \end{align*}
 The objects attached in Lemma \ref{critset} are:
 \begin{align*}
 {}& Z = {\rm Gr}_1(\CC^{A_1}) \times {\rm Gr}_1(\CC^{A_2}),\\
 {}& \tilde{\beta}= \frac{3}{2}(e_1+e_3) + \frac{5}{3}(e_2+e_4+e_5)- x_0.
 \end{align*}
 Here ${\rm Gr}_1(\CC^{A_1})$ and ${\rm Gr}_1(\CC^{A_2})$ are the orbits of ${\rm Diag}(1,0, 2,0,0)$ and ${\rm Diag}(0, 1,0, 2,2)$ under
 the unitary groups ${\rm U} (\CC^{A_1})$ and ${\rm U} (\CC^{A_2})$ respectively. 
Take $(V_1, V_2)\in Z$ arbitrarily, where $\dim V_1=2$ and $V_2=V_1^\perp$ (orthogonal complement in $\CC^5$). Furthermore,
$$V_1=V_{11}\oplus V_{12} \ {\rm and} \ V_2=V_{21}\oplus V_{22},$$
where $V_{11}\subset \CC^{A_1},$ $V_{12}\subset \CC^{A_2}$ are 1-dimensional, $V_{21}$ is the orthogonal complement
of $V_{11}$ in $\CC^{A_1}$, and $V_{22}$  the orthogonal complement
of $V_{12}$ in $\CC^{A_2}$. 
Note that
\begin{align} \label{tvb} {}& T_{(V_1, V_2)}{\rm Gr}_2(\CC^5)= {\rm Hom}_\CC(V_1, V_2) \\  \nonumber
{}&= {\rm Hom}_\CC(V_{11}, V_{21}) \oplus {\rm Hom}_\CC(V_{11}, V_{22}) 
\oplus {\rm Hom}_\CC(V_{12}, V_{21}) 
\oplus {\rm Hom}_\CC(V_{12}, V_{22}).\end{align}

 For any $v\in \CC^{A_1}$ and $\tau \in \RR$,
   $$\exp(\tau \tilde{\beta}) v= e^{\sqrt{-1}\tau(\frac{3}{2}-\frac{8}{5})}v.$$
   Similarly, for any $v\in \CC^{A_2}$ and $\tau \in \RR$,
  $$\exp(\tau \tilde{\beta}) v= e^{\sqrt{-1}\tau(\frac{5}{3}-\frac{8}{5})}v.$$ 
  Thus for $f\in {\rm Hom}_\CC(V_{11}, V_{22})$,
  $$\tilde{\beta}.f= {\sqrt{-1}\left(\frac{5}{3}-\frac{3}{2}\right)}f$$
  and for $g\in {\rm Hom}_\CC(V_{12}, V_{21})$,
  $$\tilde{\beta}.g= {\sqrt{-1}\left(\frac{3}{2}-\frac{5}{3}\right)}g.$$
  Also, $\tilde{\beta}$ is identically zero on the remaining two terms in the decomposition (\ref{tvb}). 
  Let $Y$ be the stable manifold of $\tilde{\mu}_{\tilde{\beta}}$ along $Z$. Since $\frac{3}{2}<\frac{5}{3}$, a direct complement of $T_{(V_1, V_2)} Y$ in $T_{(V_1, V_2)}({\rm Gr}_2(\CC^5))$ is just 
  ${\rm Hom}_\CC(V_{12}, V_{21})$, whose dimension is equal to 2. Thus the conclusion of Prop.~\ref{codimc} is no longer valid.
  
  To achieve the result one should change the presentation  of the orbit and think of it as ${\rm U}(5).{\rm Diag}(2,2,2,1,1).$
 The element of the orbit we were looking at before is then $(V_2, V_1)$, and the  tangent space required by Prop.~\ref{codimc} is 
 ${\rm Hom}_\CC(V_2, V_1)$, along with its canonical structure of complex vector space.   But this clearly differs from ${\rm Hom}_\CC(V_1, V_2)$, which was used in the computations above. As real vector spaces, these two spaces are naturally identified, but the complex structures induced by the two presentations are conjugate to each other, which makes that the corresponding weights of the infinitesimal generator above get opposite signs.
   \end{rem}
 
\begin{cor}\label{levelconn} If $k_i \ge 2$ for all $1\le i \le r$, then the inclusion $\mu^{-1}(x_0)\hookrightarrow {\rm F}\ell^\CC$ is 3-connected.
In particular, $\mu^{-1}(x_0)$ is a simply connected  subspace of $\Fl^\CC$. 
\end{cor}

\begin{proof}
The proof goes exactly like for \cite[Prop.~3.4]{Ma2}. For the sake of completeness we reproduce the details, as follows. 
Recall that in the stratification (\ref{stratif}), the  stratum \(S_{\A_0,[\vk_0]}\) contains
$
\tilde\mu^{-1}(0)=\mu^{-1}(x_0)
$
as a deformation retract.
We will prove that the inclusion
$
S_{\A_0,[\vk_0]}\hookrightarrow \Fl^{\mathbb C}
$
is 3-connected.

Let \((A_1,[\vk_1]),\ldots,(A_{q_1},[\vk_{q_1}])\) be the minimal elements of {\bf B}. By the closure property of the stratification, the corresponding strata $
S_{\A_1,[\vk_1]},\ldots,S_{\A_{q_1},[\vk_{q_1}]}
$ 
are closed submanifolds of \(\Fl^{\mathbb C}\). Since they are non-open strata, Prop.~\ref{codimc} shows that each of them has codimension at least 4. Being also pairwise disjoint,  by \cite[Thm.~1.11]{BW}, the inclusion
$$
\Fl^{\mathbb C}\setminus
\bigcup_{j=1}^{q_1}S_{\A_j,[\vk_j]}
\hookrightarrow
\Fl^{\mathbb C}
$$
is \(3\)-connected.

Consider now 
$$
{\bf B}_1:={\bf B}\setminus
\{(\A_1,[\vk_1]),\ldots,(\A_{q_1},[\vk_{q_1}])\},
$$
and let
$
(\A_1',[\vk_1']),\ldots,(\A_{q_2}',[\vk_{q_2}'])
$
be the minimal elements of ${\bf B}_1$. By the closure property of the stratification, the corresponding strata are closed in the manifold
$$
\Fl^{\mathbb C}\setminus
\bigcup_{j=1}^{q_1}S_{\A_j,[\vk_j]}.
$$
Again, Prop.~\ref{codimc} shows that each $S_{\A_j',[\vk_j']}$ above has codimension at least \(4\), and hence \cite[Thm.~1.11]{BW} implies that the inclusion obtained by removing these strata is \(3\)-connected.

Continuing in this way, and using the finiteness of {\bf B}, we obtain a sequence of subspaces
$$
\Fl^{\mathbb C}=M_0\supset M_1\supset\cdots\supset M_N
$$
such that every inclusion
$
M_{j+1}\hookrightarrow M_j
$
is \(3\)-connected, and
$
M_N=S_{\A_0,[k_0]}.
$
Indeed, the process terminates with  $\{(\A_0,[\vk_0])\}$, since every proper subset of {\bf B} containing $(\A_0,[\vk_0])$  has it as the greatest element and hence the latter cannot be minimal unless it is the only element of the subset.

It follows that the composite inclusion
$
S_{A_0,[\vk_0]}
\hookrightarrow
\Fl^{\mathbb C}
$
is \(3\)-connected.
Finally, \(S_{\A_0,[\vk_0]}\) contains
$
\tilde\mu^{-1}(0)=\mu^{-1}(x_0)
$
as a deformation retract. Therefore the inclusion
 $
\mu^{-1}(x_0)\hookrightarrow S_{\A_0,[\vk_0]}
$
is a homotopy equivalence, and hence
$
\mu^{-1}(x_0)\hookrightarrow \Fl^{\mathbb C}
$
is \(3\)-connected.

Since \(\Fl^{\mathbb C}\) is simply connected, the latter also implies that \(\mu^{-1}(x_0)\) is simply connected.
\end{proof}

\noindent{\it Proof of Thm.~\ref{thmconn2}}. To make the proof suitable to the considerations in this section,  
it is convenient to make a slight change and consider 
$$S={\rm Diag} (\underbrace{1 , \dots, 1}_{k_{r-1}},
\dots,
\underbrace{{r-1}, \dots, {r-1}}_{k_1}).$$
We will use Thm.~\ref{maininstr} (b) for this choice of $S$ and frames whose columns have pairwise equal norms. The relevant complex flag manifold is the ${\rm U}(n)$-orbit of 
$${\rm Diag} (\underbrace{1 , \dots, 1}_{k_{r-1}},
\dots,
\underbrace{{r-1}, \dots, {r-1}}_{k_1}, \underbrace{{0}, \dots, {0}}_{n-k}).$$
But $n-k=k_r$ and consequently the above orbit is the same as  
$${\rm U}(n).{\rm Diag} ( \underbrace{{0}, \dots, {0}}_{k_r}, \underbrace{1 , \dots, 1}_{k_{r-1}},
\dots,
\underbrace{{r-1}, \dots, {r-1}}_{k_1}),$$
which can be expressed as\footnote{Simply observe that the ${\rm U}(n)$-orbit does not change by permuting the entries of the diagonal matrix.}   
$$-{\rm I}_n+{\rm U}(n).{\rm Diag} ( \underbrace{1, \dots, 1}_{k_r}, \underbrace{2 , \dots, 2}_{k_{r-1}},
\dots, \underbrace{{r}, \dots, {r}}_{k_1})=-{\rm I}_n+\Fl^\CC.$$
The frames in ${\mathcal F}^\CC(S, n)$ are of type $F=[f_1 | \ldots | f_n]$ with
$$\|f_1\|^2 = \cdots = \|f_n\|^2,$$
the common value of these numbers being $$c:=\frac{{\rm Tr}(S)}{n}=
\frac{1\cdot k_{r-1} + \cdots (r-1)\cdot k_1}{n}.$$ 
Let us now compute $\mu_\CC^{-1}(c, \dots , c)$. 
Note that for any $X\in {\rm F}\ell^\CC$, 
$$\mu_\CC (-{\rm I}_n +X) =(-1, \ldots , -1) +\mu(X),$$
where $\mu:{\rm F}\ell^\CC \to \RR^n$ is the matrix-diagonal projection which was used in this section. (As already pointed out,
$\mu$ is the same as $\mu_\CC$.)  
Thus $\mu_\CC(-{\rm I}_n +X)=(c, \ldots, c)$ iff
$$\mu(X)=(c+1, \ldots, c+1).$$
But
\begin{align*}{}&c+1=\frac{(1\cdot k_{r-1} + \cdots (r-1)\cdot k_1) +n}{n}\\
{}& =\frac{(1\cdot k_{r-1} + \cdots (r-1)\cdot k_1) +(k_1+ \cdots +k_r)}{n}\\
{}& =\frac{k_r+ 2\cdot k_{r-1} + \dots + r\cdot k_1}{n}\end{align*}
and hence $(c+1, \ldots, c+1) = x_0$, the barycenter we used above. 
This implies that $$\mu_\CC^{-1}(c, \dots , c)=-{\rm I}_n +\mu^{-1}(x_0),$$
and consequently, by Cor.~\ref{levelconn}, the embedding of  $\mu_\CC^{-1}(c, \dots , c)$ in the orbit $-{\rm I}_n +{\rm F}\ell^\CC$ is a 3-connected map.
It only remains to use Thm.~\ref{maininstr} (b) to finish the proof. $\square$

\section{Real flag manifolds} \label{realflag}

In this section we will prove Theorems \ref{fdn} and \ref{thmconn1}. 

\subsection{Connectedness of ${\mathcal F}^\RR(S, n)$}\label{sectreal1} The  main instrument is a ``real" version of  Cor.~\ref{levelconn}. 
Again, $n\ge 2$, $r\ge 2$, and $k_1, \ldots, k_r\ge 1$ are integers such that
$k_1+ \cdots + k_r=n$. 
This time we consider the associated {\it real} flag manifold ${\rm F}\ell^\RR,$ which is the space of all symmetric $n\times n$ matrices with real entries which are ${\rm O}(n)$-conjugate with
$${\rm Diag} (\underbrace{r, \dots, r}_{k_1},
\underbrace{r-1, \dots, r-1}_{k_2},
\dots,
\underbrace{1, \dots, 1}_{k_r}).$$
We approach it from the viewpoint of a real symplectic manifold. More precisely, we denote by $\sigma$ the automorphism of ${\rm Herm}_n$ given by 
$\sigma(X)=\overline{X}$, for all $X\in {\rm Herm}_n$, where $\overline{X}$ is the entry-wise complex conjugate of $X$. It clearly induces an automorphism of $\Fl^\CC$ 
whose 
fixed point set is $\Fl^\RR$. As before, $T$ is the space of all diagonal matrices in ${\rm U}(n)$. Define $\phi:T \to T$,  $\phi(g)= g^{-1}$, for all $g\in T$. 
We are going to use the notion of {\it real Hamiltonian system}, for which we refer to \cite[Definition 3]{BaHe}. This notion had been previously defined and investigated 
in \cite{SO}, see also the survey paper \cite{S}. 

\begin{lem} \label{realstr} The manifold $\Fl^\CC$ equipped with the Kirillov-Kostant-Souriau symplectic form, the $T$-action, the moment map $\mu_\CC - x_0$ and the automorphisms 
$\sigma$ and $\phi$ above is a real Hamiltonian system.
\end{lem}

\begin{proof} One has to recall that the symplectic structure on $\Fl^\CC$ becomes apparent only in terms of the two identification mentioned at the beginning of 
Sect.~\ref{sect:critmu}. The passage from Hermitean to skew-Hermitean matrices by means of the multiplication by $\sqrt{-1}$ makes $\Fl^\CC$ into the adjoint orbit of an element 
in ${\frak u}(n)$ which is transformed by $\sigma$ into its negative.  Consequently, the objects mentioned in the statement above fit into the situation described by \cite[Example 1.9]{SO}, see also \cite[Example 1.9]{S}, the only difference 
being the moment map, which in our case is not merely $\mu_\CC$. Therefore, it only remains to verify that
\begin{align*}
{}& \tilde{\mu} \circ \sigma = -\phi_* \circ \tilde{\mu},
\end{align*} 
where $\phi_*$ is the differential of $\phi$ at the identity element of $T$. 
This is equivalent  to
$$\mu_\CC\circ \sigma - x_0 = -(\phi_* \circ \mu_\CC - \phi_*( x_0)),$$
which is clearly true, first because $\mu_\CC\circ \sigma = -\phi_* \circ \mu_\CC$ and second because $\phi_*$ transforms every element of ${\mathfrak t}$ into its negative. 
\end{proof}

Due to the result in the lemma, one can now proceed as follows. Consider the restriction of $f$ to $\Fl^\RR$, that is, $f_r:  \Fl^\RR \to \RR$, $f_r(V_\bullet) = \|\mu_\RR(V_\bullet) - x_0\|^2$, for all
$V_\bullet \in \Fl^\RR$. By  \cite[Prop.~3]{BaHe} combined with Sec.~\ref{complex} above, we have:
\begin{itemize}
\item[(a)] The critical set of $f_r$ is  
$${\rm Crit}(f_r) = \bigsqcup_{(\A,[\vk])\in {\bf B}} Z^\sigma_{{\mathcal A}, \vk}\cap \tilde{\mu}^{-1}(\tilde{\beta}_{{\mathcal A}, \bar{k}}),$$
where  $\tilde{\beta}_{{\mathcal A}, \bar{k}}$ is given by Eq.~(\ref{betatil})  and
 $$Z^\sigma_{{\A,[\overline{k}]}}= \bigcup_{\overline{\ell}\sim\vk}( \prod_{s=1}^m {\rm F}\ell^\RR_{\ell_{1,s}, \ell_{2,s}, \ldots, \ell_{r,s}}(\RR^{A_s})).$$ 
\item[(b)] There exists a stratification
 $${\rm F}\ell^\RR = \bigsqcup_{(\A, [\vk]) \in {\bf B}}S^\sigma_{\A, \vk}$$
 where {\bf B} was defined immediately after Lemma \ref{critset},  with the following properties:
 \begin{itemize}
 \item[$\bullet$]  For all $(\A, [\vk]) \in {\bf B}$, $S^\sigma_{\A, [\vk]}= S_{\A, [\vk]} \cap \Fl^\RR$, this being a submanifold of ${\rm F} \ell^\RR$ whose codimension is half the codimension of $S_{\A, [\vk]}$ in $\Fl^\CC$.  
\item[$\bullet$] For all $(\A, [\vk]) \in {\bf B}$, $S^\sigma_{\A, [\vk]}$ contains $Z^\sigma_{{\A,[\overline{k}]}} \cap \tilde{\mu}^{-1}(\tilde{\beta}_{\A,\overline{k}})$ as a deformation retract.
 \item[$\bullet$] Relative to the partial order mentioned in Sect.~\ref{complex},  the closure of any  stratum $S^\sigma_{\A, [\vk]}$   is contained in the union of that stratum and all strata whose indices precede $(\A, [\vk])$ in that order. 
 \item[$\bullet$] There exists exactly one stratum of  codimension equal to 0, namely $S^\sigma_{\A_0, [\vk_0]}$, see Rem.~\ref{betaz}.  It contains $\mu_\RR^{-1}(x_0)$ as a deformation retract. 
 \end{itemize}
\end{itemize}

Assume that $k_i\ge 2$ for all $1\le i\le r$. By Prop.~\ref{codimc}, the codimension of each stratum $S^\sigma_{\A,[\vk]}$ is either 0 or at least 2. We claim that the inclusion
$
\mu_\RR^{-1}(x_0)\hookrightarrow \Fl^\RR
$
is 1-connected.

Indeed, equip the indexing set {\bf B} with the partial order described in Sect.~\ref{complex}, so that the unique stratum of codimension 0 is indexed by the greatest element. Consider first the minimal elements of {\bf B}. By the closure property of the stratification, the corresponding strata are closed in $\Fl^\RR$. Since they are non-open, each has codimension at least 2. Hence, by \cite[Theorem 1.11]{BW}, removing these strata induces a 1-connected inclusion.

We now repeat the argument for the minimal elements of the remaining indexing set. After finitely many steps, since {\bf B} is finite, only the unique codimension-zero stratum remains. Consequently, its inclusion in $\Fl^\RR$ is 1-connected. By the properties of the stratification listed above, this stratum contains
$
\mu_\RR^{-1}(x_0)
$
as a deformation retract. Therefore the inclusion map
$
\mu_\RR^{-1}(x_0)\hookrightarrow \Fl^\RR
$
is 1-connected.
The proof of Thm.~\ref{thmconn1} now follows exactly as the proof of Thm.~\ref{thmconn2}, see the end of Sect.~\ref{complex}.

\subsection{Connectedness of ${\mathcal F}^\RR_{{\rm I}_k, \vec{d}}$}
The same method  can now be used to prove Thm.~\ref{fdn}. Since $S={\rm I}_k$, the relevant flag manifold, as prescribed in Sect.~\ref{sec:mainres}, is the ${\rm O}(n)$-orbit of the $n \times n$ matrix
$$\begin{pmatrix}
{\rm I}_k& 0\\
0 &0
\end{pmatrix}.$$
This is naturally identified with a real Grassmannian, and we denote it by ${\rm Gr}_k(\RR^n)$. We also need the ${\rm U}(n)$-orbit of the same matrix, which we identify with ${\rm Gr}_k(\CC^n)$.  Just like in Lemma \ref{realstr}, the latter Grassmannian, 
the canonical $T$-action, the map $\mu_\CC-\vec{d}$, and the same two involutive automorphisms make together a real Hamiltonian system. 
The key ingredient is again the Kirwan-Morse stratification, this time induced by the function
$
V\longmapsto \|\mu_\CC(V)-\vec{d}\|^2,
$
$V\in {\rm Gr}_k({\mathbb R}^n)$. 
It has the properties listed in Sect.~\ref{sectreal1} above.  Although it was already constructed in \cite[Sect.~3]{Ma1}, we would like to provide more details concerning it in terms of the framework we have been using here, since Lemma \ref{critset} is not applicable any more. Set $k_1:=k$ and $k_2:=n-k$. This time we need $\tilde{\mu} : {\rm Gr}_k(\mathbb C^n) \to \RR^n$, $\tilde{\mu}=\mu_\CC-\vec{d}$. We search for all $\tilde{\beta}\in \RR^n$ which lie in the $\tilde{\mu}$-image
of a critical component of $\tilde{\mu}_{\tilde{\beta}}$. Just like in Sec.~\ref{sect:critmu}, $\tilde{\beta}$ gives rise to a partition ${\mathcal A}$ and to a matrix
$\overline{k}$, such that the conditions in (\ref{eq:matching}) are satisfied, with $r=2$. For all $1\le s\le m$, the components of $\tilde{\beta}$ of index in $A_s$ are pairwise equal, their common value being this time equal to
$$\frac{1}{n_s}(k_{1,s}-\sum_{i\in A_s}d_{i}).$$ 
Thus
$$\tilde{\beta}= \sum_{s=1}^m \frac{k_{1,s}-\sum_{i\in A_s}d_{i}}{n_s}\sum_{j\in A_s}e_j,$$
the numbers 
$$\frac{k_{1,s}-\sum_{i\in A_s}d_{i}}{n_s}, \quad 1\le s\le m$$
being pairwise distinct. For any such $\tilde{\beta}$, the critical components of $\tilde{\mu}_{\tilde{\beta}}$ which contain $\tilde{\beta}$ in its image are of the form
$$\prod_{s=1}^m \Fl^\CC_{\ell_{1,s}, \ell_{2,s}}(\CC^{A_s})$$
where $(\ell_{i,s})_{1\le i \le 2, 1\le s\le m}$ are integers satisfying (\ref{eq:matching}) and moreover
$$\frac{\ell_{1,s}-\sum_{i\in A_s}d_{i}}{n_s}=\frac{k_{1,s}-\sum_{i\in A_s}d_{i}}{n_s}, \quad 1\le s\le m.$$
On the other hand, (\ref{eq:matching}) implies
$$\ell_{1,s}+\ell_{2,s} =k_{1,s}+k_{2,s}, \quad 1\le s\le m.$$
Consequently, $\ell_{i,s}=k_{i,s}$ for all $1\le i \le 2$ and $1\le s\le m$. Thus in this case there is exactly one critical component of $\tilde{\mu}_{\tilde{\beta}}$ which contains
$\tilde{\beta}$. Recall that for $r\ge 3$, this is in general no longer the case, see Rem.~\ref{rem:more}.

Now  the desired stratification of ${\rm Gr}_k(\RR^n)$ arises from a direct application of the results from \cite{BaHe} which were mentioned above. 
It is shown in \cite[Prop.~4.1]{Ma1}  that  every non-open stratum has codimension at least 2.
They key-result is again that the inclusion map 
$
\mu_\RR^{-1}(\vec{d})\hookrightarrow
{\rm Gr}_k({\mathbb R}^n)
$
is 1-connected: this follows from the properties of the stratification, by the procedure exposed in the previous subsection.
To conclude, one applies Thm.~\ref{maininstr} (a). 

\section{Examples}

\begin{ex} \label{ex:nonconn} As mentioned in the introduction, the space ${\mathcal F}^\RR(S, n)$ is in general not path-connected. 
For example, this happens if $k=n-1$ and $S={\rm I}_k$, see \cite[Sect.~3]{DS}.
\end{ex}
\begin{ex} \label{ex:nonconn0}
Another  example where ${\mathcal F}^\RR(S, n)$ is not path-connected can be obtained by taking
 $k=2, n=3$, and
$$S={\rm Diag}(1,3).$$ Here is a proof. 
Write $f_i=(a_i, b_i)^t$, $1\le i \le 3$. The corresponding frame space is the subspace of $\RR^6$ determined by the equations
\begin{align*}
{}& a_i^2+b_i^2= \frac{4}{3}, \ {\rm for \ all } \ 1\le i \le 3, \\
{}& a_1^2+a_2^2+a_3^2=1, \\
{}& b_1^2+b_2^2+b_3^2=3, \\
{}& a_1b_1+a_2b_2+a_3b_3=0.
\end{align*}
Observe that whenever these conditions are satisfied, $b_1 \neq 0$.
This is because
$$b_1^2=3-b_2^2-b_3^2\ge 3-2\cdot \frac{4}{3}=\frac{1}{3}.$$
On the other hand, let us pick $F_0\in {\mathcal F}^\RR(S, 3)$. Clearly, $-F_0\in {\mathcal F}^\RR(S, 3)$ too. 
Since the function which associates to $F\in {\mathcal F}^\RR(S, 3)$ the component $b_1$ is continuous, we deduce that ${\mathcal F}^\RR(S, 3)$ is not connected, and thus not path-connected. 
\end{ex}

\begin{ex} \label{ex:nonconn1} The space ${\mathcal F}^\RR(S, n)$ can be path-connected even for $k=n-1$ (cf.~Ex.~\ref{ex:nonconn}), by choosing $S$ appropriately. To see this,
take again $k=2, n=3$, and this time
$$S={\rm Diag}(1,2).$$
With the notations established above, we are given six real numbers $a_i, b_i$ for $i = 1, 2, 3$. Let us denote these as  vectors in $\mathbb{R}^3$:
\[
\vec{a} = (a_1, a_2, a_3) \quad \text{and} \quad \vec{b} = (b_1, b_2, b_3).
\]
The frame space is determined by the following constraints:
\begin{enumerate}
    \item[$\bullet$] $a_i^2 + b_i^2 = 1$ for all $i \in \{1, 2, 3\}$
    \item[$\bullet$] $\|\vec{a}\|^2 = \sum_{i=1}^3 a_i^2 = 1 \implies \|\vec{a}\| = 1$
    \item[$\bullet$] $\|\vec{b}\|^2 = \sum_{i=1}^3 b_i^2 = 2 \implies \|\vec{b}\| = \sqrt{2}$
    \item[$\bullet$] $\vec{a} \cdot \vec{b} = \sum_{i=1}^3 a_i b_i = 0$.
\end{enumerate}
We will prove as follows:
\begin{enumerate}
    \item[{}] - At least one of the components $a_i$ must be equal to $0$.
    \item[{}] - The resulting solution subspace of $\mathbb{R}^6$ is path-connected.
\end{enumerate}

{\it Step 1}: At least one of $a_1, a_2$, and $a_3$  is equal to 0.

Since $\vec{a}$ and $\vec{b}/\sqrt{2}$ are orthonormal, the vector
\[
\vec{c} = \vec{a} \times \frac{\vec{b}}{\sqrt{2}}
\]
is the third vector completing them to an orthonormal basis.
Hence the matrix
\[
\begin{pmatrix}
a_1 & a_2 & a_3 \\
\frac{b_1}{\sqrt{2}} & \frac{b_2}{\sqrt{2}} & \frac{b_3}{\sqrt{2}} \\
c_1 & c_2 & c_3
\end{pmatrix}
\]
is in ${\rm O}(3)$.
 Consequently, its columns are also  orthonormal and thus
\[
a_i^2 + \frac{b_i^2}{2} + c_i^2 = 1, \quad 1\le i \le 3.
\]
Substitute  $a_i^2=1-b_i^2$ and  get:
\[
c_i^2  = \frac{b_i^2}{2},
\]
$1\le i \le 3$. 
Evaluating the first component of the cross product $\vec{c} = \vec{a} \times \frac{\vec{b}}{\sqrt{2}}$ yields:
\[
c_1 = \frac{1}{\sqrt{2}}(a_2 b_3 - a_3 b_2).
\]
Squaring both sides and substituting $c_1^2 = \frac{b_1^2}{2}$ leads to:
\[
 b_1^2 = a_2^2 b_3^2 + a_3^2 b_2^2 - 2a_2 a_3 b_2 b_3.
\]
Substitute $b_1^2 = 1 - a_1^2 = a_2^2 + a_3^2$ on the left side, and $b_i^2 = 1 - a_i^2$ on the right side:
\begin{align*}{}&
a_2^2 + a_3^2 = a_2^2(1 - a_3^2) + a_3^2(1 - a_2^2) - 2a_2 a_3 b_2 b_3\\
{}& = a_2^2 + a_3^2 - 2a_2^2 a_3^2 - 2a_2 a_3 b_2 b_3,
\end{align*}
and thus 
\[
 a_2 a_3(a_2 a_3 + b_2 b_3) = 0.
\]
By cyclic permutations of indices  we obtain the following three equations:
\begin{align*}
    a_2 a_3(a_2 a_3 + b_2 b_3) &= 0 \\
    a_3 a_1(a_3 a_1 + b_3 b_1) &= 0 \\
    a_1 a_2(a_1 a_2 + b_1 b_2) &= 0.
\end{align*}
Assume now that $a_1 \neq 0$, $a_2 \neq 0$, and $a_3 \neq 0$. Then clearly
\[
b_2 b_3 = -a_2 a_3, \quad b_3 b_1 = -a_3 a_1, \quad b_1 b_2 = -a_1 a_2.
\]
Multiplying these three equations together gives:
\[
(b_1 b_2 b_3)^2 = (-1)^3 (a_1 a_2 a_3)^2 = 
-(a_1 a_2 a_3)^2.
\]
The right-hand side is strictly negative, whereas  $(b_1 b_2 b_3)^2 \ge 0$, which produces a contradiction. 
Thus at east one component $a_i$ must be equal to $0$, as required.

 {\it Step 2}: the actual proof of  path connectedness.

For brevity, denote  ${\mathcal F}:={\mathcal F}({\rm Diag}(1,2), 3)$, which is regarded as a subspace of  $\mathbb{R}^6$. From Step 1, its projection  onto the $\vec{a}$-coordinates defines a space $B \subset S^2$, which consists of three  great circles:
\[
B = \{ \vec{a} \in S^2 : a_1=0 \} \cup \{ \vec{a} \in S^2 : a_2=0 \} \cup \{ \vec{a} \in S^2 : a_3=0 \}.
\]

Next, we analyze the allowed vectors $\vec{b}$ along  each of the three circles.
We  divide ${\mathcal F}$ according to which of $a_1, a_2$, or $a_3$ is equal to 0:

If $a_1 = 0$, then  $a_2^2 + a_3^2 = 1$ and $b_1 = \pm 1$. Because $a_2 b_2 + a_3 b_3 = 0$, we deduce that
$b_2=\pm a_3$ and $b_3=\mp a_2$.  The corresponding point in ${\mathcal F}$ is
in one of the following four subspaces:
    $$
    S_1(\epsilon_1, \epsilon_2) = \{(0, a_2, a_3,  \epsilon_1, \, \epsilon_2 a_3, \, -\epsilon_2 a_2 ) \mid a_2^2+a_3^2=1\},
    $$
where $\epsilon_1, \epsilon_2 \in \{-1,1\}$.
Similarly, if $a_2=0$, the point is in 
$$
    S_2(\delta_2, \delta_1) = \{( a_1, 0, a_3, \delta_1 a_3, \, \delta_2, \, -\delta_1 a_1 ) \mid a_1^2+a_3^2=1\},
    $$
and if $a_3=0$, it is in
$$
    S_3(\gamma_3, \gamma_1) = \{ (a_1, a_2, 0,  \gamma_1 a_2, \, -\gamma_1 a_1, \, \gamma_3 ) \mid a_1^2+a_2^2=1\},
    $$
where $\delta_1, \delta_2, \gamma_1, \gamma_3 \in \{-1,1\}$. 
We conclude that
$${\mathcal F}= \Big(\bigcup_{\epsilon_1, \epsilon_2 \in \{- 1, 1\} } S_1(\epsilon_1, \epsilon_2)\Big)
\bigcup \Big(\bigcup_{\delta_2, \delta_1 \in \{ -1, 1\}} S_2(\delta_2, \delta_1)\Big)
\bigcup \Big(\bigcup_{\gamma_3, \gamma_1 \in \{-1,1\}} S_3(\gamma_3, \gamma_1)\Big).
$$
Note that each of the twelve components in the union above is path-connected, being homeomorphic to the circle $S^1$. The union itself is path-connected provided that we dispose
of enough many intersection points. Here is a list of such points:

\begin{enumerate}
    \item Take $a_1=0, a_2=0, a_3=1$ and observe that
    \begin{align*}
        S_1(\epsilon_1, \epsilon_2) \Big|_{a_3=1, a_2=0} &= (0,0,1, \epsilon_1, \epsilon_2, 0) \\
        S_2(\delta_2, \delta_1) \Big|_{a_3=1, a_1=0} &= (0,0,1,\delta_1, \delta_2, 0).
    \end{align*}
    Setting $\delta_1 = \epsilon_1$ and $\delta_2 = \epsilon_2$ shows that  $S_1(\epsilon_1, \epsilon_2)$ and $S_2(\epsilon_2, \epsilon_1)$  intersect at the  point $(0, 0, 1, \epsilon_1, \epsilon_2, 0)$.
    
    \item 
  Take $a_1=0, a_2=1, a_3=0$ and observe that
    \begin{align*}
        S_1(\epsilon_1, \epsilon_2) \Big|_{a_2=1, a_3=0} &= (0,1,0,\epsilon_1, 0, -\epsilon_2) \\
        S_3(\gamma_3, \gamma_1) \Big|_{a_2=1, a_1=0} &= (0,1,0,\gamma_1, 0, \gamma_3).
    \end{align*}
    Thus $S_1(\epsilon_1, \epsilon_2)$ and $S_3(-\epsilon_2, \epsilon_1)$ intersect at the  point $(0, 1, 0, \epsilon_1, 0, -\epsilon_2)$.
    
    \item Take  $a_1=1, a_2=0, a_3=0$ and this time note that
    \begin{align*}
        S_2(\delta_2, \delta_1) \Big|_{a_1=1, a_3=0} &= (1,0,0, 0, \delta_2, -\delta_1) \\
        S_3(\gamma_3, \gamma_1) \Big|_{a_1=1, a_2=0} &= (1,0,0, 0, -\gamma_1, \gamma_3).
    \end{align*}
    Thus $S_2(\delta_2, \delta_1)$ and $S_3(-\delta_1, -\delta_2)$ intersect at the  point $(1, 0, 0, 0, \delta_2, -\delta_1)$.
\end{enumerate}
Connectedness of ${\mathcal F}$ now follows from the fact that the following intersections are non-empty:
\begin{align*}
{}&  S_1(1, 1) \cap  S_2(1,1), & S_2(1, 1) \cap  S_3(-1, -1), & {} &  S_3(-1, -1) \cap S_1(-1, 1),  & {}& S_1(-1, 1) \cap S_2(1, -1),  \\ 
{}& S_2(1, -1) \cap  S_3(1, -1),  &  S_3(1, -1) \cap  S_1(-1, -1),  & {}& S_1(-1, -1) \cap  S_2(-1, -1),   & {}& S_2(-1, -1) \cap  S_3(1, 1), \\
 {}& S_3(1, 1) \cap  S_1(1, -1), & S_1(1, -1) \cap  S_2(-1, 1) ,  & {}& S_2(-1, 1) \cap  S_3(-1, 1) ,  & {}& S_3(-1, 1) \cap  S_1(1, 1). 
\end{align*}
(A path from $S_1(1, 1)$ to itself going through all $S_i(\epsilon, \delta)$, $1\le i \le 3$, $\epsilon, \delta \in \{-1,1\}$ can be obtained.)
\end{ex}

\begin{ex} \label{ex:nonsconn} The space ${\mathcal F}^\CC(S, n)$ is in general not simply connected. To illustrate this, we take again 
$k=2, n=3$, and, just like in Ex.~\ref{ex:nonconn0}, 
$$S={\rm Diag}(1,3).$$
The relevant flag manifold is this time $\Fl^\CC={\rm U}(3).{\rm Diag}(1,3,0)$ together with the matrix-diagonal projection map $\mu=\mu_\CC: \Fl^\CC\to \RR^3$.
Note that $\mu$ is also a moment map corresponding to the canonical action of
the 3-torus $$\{{\rm Diag} (t_1, t_2, t_3) \mid t_1, t_2, t_3\in \CC, |t_1|=|t_2|=|t_3|=1\}$$ on $\Fl^\CC$.
 Like in Sec.~\ref{subsec:r}, (more precisely, Lemma \ref{gbundle}) one can construct a principal bundle
\begin{equation}\label{s1t}S^1\times S^1 \to  {\mathcal F}^\CC (S, 3) \to \mu^{-1}(x_0),\end{equation}
where $x_0=\left(\frac{4}{3}, \frac{4}{3}, \frac{4}{3}\right).$
Let us now consider the following piece of the corresponding long exact homotopy sequence:
$$\pi_1({\mathcal F}^\CC (S, 3) )\to \pi_1(\mu^{-1}(x_0)) \to \pi_0(S^1\times S^1).$$
Since $S^1\times S^1$ is connected,  the first map in the sequence is surjective. 
We will now show that $\pi_1(\mu^{-1}(x_0))\neq 0$, and then $\pi_1({\mathcal F}^\CC (S, 3) )$ must be non-zero as well. 
To this end, consider
the torus $$T^2=\{{\rm Diag} (t_1, t_2, t_3) \mid t_1, t_2, t_3\in \CC, |t_1|=|t_2|=|t_3|=1, t_1t_2t_3=1\}.$$
Let also consider its subgroup $\ZZ_3=\{\varepsilon {\rm I}_3 \mid \varepsilon \in \CC, \varepsilon^3=1\}$. The quotient
$T^2/\ZZ_3$ 
 acts in Hamiltonian fashion on $\Fl^\CC$, a moment map being $\mu_{T^2}:=\mu-x_0$.

\noindent {\it Claim.} The action of $T^2$ on $\mu_{T^2}^{-1}(0)=\mu^{-1}(x_0)$ is locally free and $\mu^{-1}(x_0)/T^2$ is a 2-manifold. 

\noindent Take $X \in \mu^{-1}(x_0)$ and $A={\rm Diag}(t_1, t_2, t_3)$ such that $AXA^{-1}=X$, that is,
$AX=XA$. Say that $X=(X_{ij})_{1\le i, j\le 3}$. Then
$$(AX-XA)_{ij}= (t_i-t_j)X_{ij}.$$
We cannot have $X_{12}=0$ and $X_{13}=0$ since if so, $X$ being Hermitian, we also have
$X_{21}=X_{31}=0$, and thus $X$ has $\frac{4}{3}$ as eigenvalue, which is a contradiction (the eigenvalues of $X$ are 0, 1, and 3). One deduces that $t_1=t_2$ or $t_1=t_3$. Similarly, $t_2=t_1$ or $t_2=t_3$ and also 
$t_3=t_1$ or $t_3=t_2$. Since $t_1=t_2=t_3$, the $T^2$-stabilizer of $X$ is the copy of $\ZZ_3$ defined above.  Thus the  action of $T^2/\ZZ_3$ is free, and therefore $0$ is a regular value of $\mu_{T^2}$, see \cite[Sect.~23]{ACS}. It follows that  $\mu^{-1}(x_0)/T^2$ is actually the quotient of $\mu^{-1}(x_0)$ by the {\it free} action of $T^2/\ZZ_3$ and is therefore smooth. A dimension count shows that the dimension of this quotient as a manifold is equal to 2.

\noindent By 
\cite[Theorem 1.2]{HL}, one can compute the fundamental group of the symplectic quotient as $\pi_1(\mu^{-1}(x_0)/T^2)\simeq \pi_1(\Fl^\CC)=\{0\}.$ Thus
$\mu^{-1}(x_0)/T^2$ is diffeomorphic to $S^2$. Consider the following piece of the corresponding long exact homotopy sequence:
$$\pi_2(S^2)\simeq \ZZ  \to \pi_1(T^2)\simeq \ZZ\times \ZZ  \to \pi_1(\mu^{-1}(x_0)).$$
Since the image of the first map has rank at most 1, it cannot equal $\pi_1(T^2)$. Hence the first map is not surjective. By exactness, the second map has non-zero image and therefore  $\pi_1(\mu^{-1}(x_0)) \neq 0$, as desired. 
\end{ex}

\begin{ex} \label{ex:last} We now want to illustrate Thm.~\ref{fdn}. A class of frame spaces which satisfy the assumption (\ref{eqndi})  was previously obtained by Caine, Needham, and Shonkwiler in \cite{CNS}.
To describe it, first order $d_1, \ldots, d_n$ by means of a permutation 
 $\sigma\in S_n$, that is,
 $$d_{\sigma(1)} \ge d_{\sigma(2)} \ge \cdots \ge d_{\sigma(n)}.$$
 As a special case of \cite[Thm.~4.5]{CNS},  the space ${\mathcal F}^\RR_{{\rm I}_k, \vec{d}}$ is path-connected provided that $d_1, \ldots, d_n \in {\mathbb Q}$ and 
 $$\frac{k}{n\ell}\cdot {\rm min} \{j \mid \sum_{i=1}^j d_{\sigma(i)} >\ell \} \ge \frac{k}{n}\cdot \frac{k+1}{k-1},$$
 for all $\ell =1, \ldots, k-1$. (It is also required that $k\ge 2$, but this follows using the argument below and \cite[Rem.~1.2]{Ma1}.) 
For $\ell = k-1$ one obtains
$$  {\rm min} \{ j \mid \sum_{i=1}^j  d_{\sigma(i)} > k-1\}\ge k+1.$$
This implies
$$ \sum_{i=1}^k  d_{\sigma(i)} \le k-1.$$
Since 
 $\sum_{i=k+1}^n d_{\sigma(i)} = k-\sum_{i=1}^k d_{\sigma(i)}$, we must have 
 $$ \sum_{i=k+1}^n d_{\sigma(i)}  \ge 1.$$
 Notice now that this inequality is equivalent to (\ref{eqndi}). 
 \end{ex}


\begin{thebibliography}{Fo-Ge-Po}

\bibitem{BaHe} T.~Baird and N.~Heydari, {\it Cohomology of quotients in real symplectic geometry}, Alg.~Geom.~Topology {\bf 22} (2022), 3249-3276


\bibitem{CNS} A.~Caine, T.~Needham, and C.~Shonkwiler, {\it Optimization and topology of spaces of Parseval frames}, SIAM J.~Matrix Anal.~Appl.~{\bf 47} (2026), 1375-1399

\bibitem{ACS} Ana Cannas da Silva, {\it Lectures on Symplectic Geometry}, Lecture Notes in Mathematics, vol.~1764, Springer-Verlag, Berlin, 2001



\bibitem{CMS} J.~Cahill, D.~Mixon, and N.~Strawn, {\it Connectivity and irreducibility of algebraic varieties of finite unit norm tight frames}, SIAM J.~Appl.~Algebra Geom.~{\bf 1} (2017), 38-72

\bibitem{DS} K.~Dykema and N.~Strawn, {\it Manifold structure of spaces of spherical tight frames}, Int.~J.~Pure Appl.~Math.~{\bf 28} (2006), 
 217-256
 
\bibitem{Gray} B.~Gray, {\it Homotopy Theory - An Introduction to Algebraic Topology}, Pure and Applied Mathematics Series, vol.~64, Academic Press, 1975


\bibitem{Ho} A.~Horn, {\it Doubly stochastic matrices and the diagonal of a rotation matrix}, American Journal of Mathematics {\bf 76} (1954), 620-630

\bibitem{Ki} F.~C.~Kirwan, {\it Cohomology of Quotients in Symplectic and Algebraic Geometry}, Mathematical
Notes, vol.~31, Princeton Univ.~Press, New Jersey, 1984

\bibitem{Le} E.~Lerman, {\it Gradient flow of the norm squared of a moment map}, Enseign.~Math.~{\bf 51} (2005), 117-127

\bibitem{HL} H.~Li, {\it The fundamental group of symplectic manifolds with Hamiltonian Lie group actions}, J.~Symplectic Geom.~{\bf 4} (2006), 345-372 



\bibitem{Ma1} A.-L.~Mare, {\it Connectivity properties of the Schur-Horn map for real Grassmannians}, Abh.~Math.~Semin.~Univ.~Hambg.~{\bf 94} (2024), 33-55

\bibitem{Ma2} A.-L.~Mare, {\it Simply connectedness of spaces of tight frames}, Beitr.~Algebra Geom., in press,
DOI: https://doi.org/10.1007/s13366-026-00859-5

\bibitem{NS} T.~Needham and C.~Shonkwiler, {\it Symplectic geometry and connectivity of spaces of frames},  Adv.~Comput.~Math.~{\bf 47} (2021), no.~1, 5

\bibitem{SO} L.~O’Shea and R.~ Sjamaar, {\it Moment maps and Riemannian symmetric pairs},
Math.~Ann.~{\bf 317} (2000), 415-457 



\bibitem{Schur} I.~Schur, {\it \"Uber eine Klasse von Mittelbildungen mit Anwendungen auf die Determinantentheorie}, 
Sitzungsber.~Berl.~Math.~Ges.~{\bf 22} (1923), 9-20

\bibitem{S} R.~Sjamaar, {\it  Real symplectic geometry},
Afr.~ Diaspora J.~Math.~{\bf 9} (2010), 34-52



 

\bibitem{Wa} S.~Waldron, {\it An Introduction to Finite Tight Frames},
Applied and Numerical Harmonic Analysis, Birkh\"auser/Springer, New York, 2018


\bibitem{BW} B.~Williams, {\it Connectivity of manifold complements}, preprint available at
\url{https://personal.math.ubc.ca/~tbjw/ConnectivityOfManifoldComplements.pdf}

\end{thebibliography}
\end{document}